\documentclass[smallextended,referee,envcountsect]{svjour3}

\usepackage{amsfonts,amssymb,amsmath}
\usepackage{exscale}
\usepackage[weather,alpine,misc,geometry]{ifsym}
\usepackage{color}
\usepackage[all]{xy}
\usepackage{empheq}

\usepackage{xcolor}

\usepackage[ps2pdf]{hyperref}  
\hypersetup{
    colorlinks=true,
    linkcolor=blue, 
    citecolor=red, 
    urlcolor=blue  } 

\usepackage[numbers]{natbib}

\usepackage{mathptmx}

\usepackage[applemac]{inputenc}
\definecolor{myblue}{rgb}{.8, .8, 1}
  
\definecolor{RedOrange} {cmyk}{0,0.77,0.87,0}

\def\smartqedtr{\def\qedtr{\ifmmode\triangle\else{\unskip\nobreak\hfil
\penalty50\hskip1em\null\nobreak\hfil$\triangle$
\parfillskip=0pt\finalhyphendemerits=0\endgraf}\fi}}
\smartqedtr  

\def\gph{\mathop{\rm gph\,}}

\def\inte{\mathop{\rm int }}
\def\dom{\mathop{\rm Dom\,}}

\def\reg {\mathop{\rm reg\,}}
\def\cone{\mathop{\rm cone\,}}

\def\N{{\mathbb{N}}}
\def\T{{\mathbb{T}}}
\def\R{{\mathbb{R}}}
\def\C{{\mathcal {C}}}
\def\C{{\mathcal {C}}}
\def\Q{{\mathcal {Q}}}
\def\P{{\mathcal {P}}}

\lccode`\-=`\-


\renewcommand {\theenumi} {\roman{enumi}}

\begin{document}

\title{Directional H\"older Metric Regularity
\thanks{Research partially supported by
Ministerio de Economıa y Competitividad under grant
MTM2011-29064-C03(03),  by LIA  `` FormathVietnam " and by NAFOSTED.}}
\author{Huynh Van Ngai \and Nguyen Huu Tron \and Michel Th\'era}
\institute{Van Ngai Huynh  \at
Department of Mathematics, University of
Quy Nhon, 170 An Duong Vuong, Quy Nhon, Vietnam \\
\email{ngaivn@yahoo.com}
\and
Huu Tron Nguyen
 \at
Department of Mathematics, University of
Quy Nhon, 170 An Duong Vuong, Quy Nhon, Vietnam \\
\email{nguyenhuutron@qnu.edu.vn}
\and
Michel Th\'era  (\Letter\,) \at
Laboratoire XLIM, Universit\'e de Limoges and Federation University Australia \\
\email{michel.thera@unilim.fr, m.thera@federation.edu.au}
}


\date{Received: date / Accepted: date}

\maketitle

\begin{abstract}
This paper sheds new light on  regularity of multifunctions  through various characterizations of directional H\"{o}lder/Lipschitz metric regularity,   which are based on the concepts of  slope and coderivative. By using these   characterizations, we show that  directional H\"{o}lder/Lipschitz metric regularity is stable,  when the multifunction under consideration is perturbed suitably. Applications of  directional H\"{o}lder/Lipschitz metric regularity to investigate the stability and the sensitivity analysis  of parameterized optimization problems are  also discussed.
\end{abstract}



\section{Introduction}

Throughout the last  three decades, metric regularity has been significantly developed and has become  one of  the central concepts of modern variational analysis.The terminology ``metric regularity'' was coined by  Borwein \cite {BorZhuang88},  but  the roots of this notion can be traced back to the classical \emph{open mapping theorem} and its subsequent generalization to nonlinear mappings known as the  \emph{Lyusternik-Graves theorem}. The theory of metric regularity is extraordinary useful for  investigating  the behavior of  solutions of a nonlinear equation under small perturbations of the data, or more generally the behavior of the solution set of generalized
equations associated with  a set-valued mapping.  As a result, metric regularity plays an important role in many aspects of optimization, differential inclusions, control theory, numerical methods and in many problems of analysis.
 According to the long history of metric regularity there is an abundant  literature on conditions ensuring this property. We refer the reader  to  the  basic  monographs  \cite {Kum-Kl, DR, B.book1, B.book2, JPP}, to  the excellent survey of A. Ioffe \cite{ALEX} (in preparation)
  and to some (non exhaustives) references
 \citep  {BorZhuang88, AzeSMAI, Aze06, RefBonS, BD,BorZhu96,  RefCom,  RefDT1, DmiKru09.1, F-90, FQ-12, Io00, RefIo2,  RefJT, JT1, Lyusternik,  RefMorS, RefNT3, 
Pen89}.
\vskip 0.1cm
 Apart from the study of the  usual (Lipschitz) metric regularity,  H\"{o}lder metric regularity or more generally nonlinear metric regularity have been  studied over  the  years $1980-1990$s  by several authors,   including  for example Borwein and Zhuang \cite{BorZhuang88}, Frankowska \cite{F-90}, Penot \cite{Pen89}, and recently, for  instance,  Frankowska and Quincampoix \cite{FQ-12}, Ioffe \cite{I-Hreg}, Li and  Mordukhovich \cite{LiMo}, Oyang and Mordukovhich \cite{OM}.
\vskip 0.2cm
Recently, several directional versions of metric regularity  notions were considered. In \citep{ AruAvaIzm07, AI}, Arutyunov et al have introduced and studied a notion of \textit{directional metric regularity}. This notion is an extension of an earlier notion used by Bonnans and  Shapiro \cite{RefBonS} to study sensitivity analysis. Later, Ioffe \cite{I-reg-concept} has introduced and investigated an extension called {\it relative metric regularity} which covers many notions of metric regularity in the literature. In particular,  another version of directional metric regularity/subregularity has been introduced and extensively studied by Gfrerer in \citep{Gfrerer-SVA,Gfr2} where some variational characterizations of this 
concept have been established and
successfully applied
 to study optimality conditions for mathematical programs. In fact, this directional regularity property has been earlier used by Penot \cite {Penot-SIOPT} to study second order optimality conditions. In the line of the directional version of  metric subregularity considered by Gfrerer  \cite{Gfrerer-SVA},    Huynh, Nguyen and Tinh \cite{NTrT} have  studied directional H\"{o}lder metric subregularity in order  to investigate  tangent cones to zero sets in degenerate cases.
\vskip 0.1cm
It is our aim in  the  present article  to study  a  \textit{directional version} of H\"{o}lder metric regularity.
The structure of the article is as follows. In the Section 2, we establish slope-based characterizations of directional  H\"{o}lder/Lipschitz metric regularity. In Section 3, we present a stability property for
  directional  Lipschitz metric regularity. In Section 4, a sufficient condition for  directional  H\"{o}lder/Lipschitz metric regularity based on the Fr\'{e}chet coderivative  is established in Asplund spaces. This condition becomes necessary when,  either the multifunction under consideration is convex,  or  when considering  directional Lipschitz metric regularity. It was silmultaneously showed that under this condition,  directional  H\"{o}lder/Lipschitz metric regularity persists  when the multifunction is perturbed  by a  Hadamard differentiable mapping.
Applications to  the study of  the stability and the sensitivity analysis of parameterized optimization problems are discussed in  Section 4. The last section contains concluding remarks.
\section{Notations and Preliminaries}
Throughout we let $X$ and $Y$ denote  metric spaces endowed with  metrics both denoted by $d(\cdot,\cdot).$ We denote
the  open and  closed balls with center $x$  and radius $r>0$  by $B(x,r)$  and $\bar{B}(x,r),$ respectively. For  a given set $C$, we write  $\inte C$ for its topological interior.
A set-valued mapping (also called multifunction) $F: X\rightrightarrows Y$ is a mapping assigning,  to each point $x\in X$, a subset (possibly empty) $F(x)$ of $Y$. We use the notations
 $$\gph F:= \{(x,y)\in X\times Y\ :\ y\in F(x)\} \quad  \text{and }\quad
\dom F:=\{x\in X\ :\ F(x)\neq \emptyset \} $$
 for  the \textit{graph} of  and the  \textit{domain}  of $F$,  respectively. For  each set-valued mapping  $F: X\rightrightarrows Y$,  we define the \textit{inverse } of $F$,  as  the mapping $F^{-1}:Y\rightrightarrows X$  defined by
  $F^{-1}(y):=\{x\in X\ :\ y\in F(x)\}, \,\, y\in Y\}$ and satisfying
 $$(x,y)\in \gph F\;\iff\; (y,x)\in \gph F^{-1}.$$
 We use the standard notation $d(x, C)$ to denote  the distance  from $x$ to a
set $C$ ; it is   defined by
 $d(x, C) = \inf_{z \in C} d(x,z)$,  with the convention that   $d(x, S) = +\infty$ whenever $S$ is  the empty set.
As pointed out in the introduction,  main attention in  this contribution  is paid to the study of the concept of metric regularity. Recall that a mapping $F$ is said to be {\it metrically regular} at  $(\bar x,\bar y)\in \gph F$   with modulus  $\tau>0$,  if there
exists a neighborhood $U\times V$ of $(\bar x,\bar y)$  such that
\begin{equation}\label{Regular}
d(x,F^{-1}( y))\leq \tau d(y,F(x))\quad\mbox{for all} \quad (x,y)\in
U\times V.
\end{equation}
In other words,  metric regularity allows to estimate the  dependence of the distance
of a trial point $x\in X$  from the solution set $ F^{-1}(y) $ in terms of the residual quantity $d(y, F(x))$  for all pairs
$(x, y)$  around the reference pair $(\bar x, \bar y) \in \gph F$.
 The infinum of all moduli $\tau$  is denoted by $\reg F(\bar x,\bar y).$\\
 If in the above   definition  we fix  $y=\bar y$ in (\ref{Regular}),  then we obtain  a weaker notion  called  \textit{metric subregularity},  see e.g. \citep{ B.book1, JPP, Roc-Wet}.
Observe  that this latter property is equivalent to the existence of some neighborhood $U$ of $\bar x$ such that
 $$d(x, F^{-1}(\bar y)) \leq  \tau d(\bar y, F(x)) \quad \text{ for all} \quad x\in U.$$ It is also well known that metric subregularity
 can be treated in the framework of the theory of error bounds of
extended-real-valued functions, see e.g. \citep{ Kr15.2,Kr15}.
\vskip 2mm
An H\"older version of metric regularity is defined as follows (Frankowska and Quincampoix \cite{FQ-12}, Ioffe \cite{I-Hreg}).
Let  $q\in(0,1]$  be given. A mapping $F$ is said to be {\it
metrically $q$-regular} or H\"older metrically regular of
order $q$ at  $(\bar x,\bar y)\in \gph F $   with modulus  $\tau>0$ if there exists a neighborhood
$U\times V$ of $(\bar x, \bar y)$  such that
\begin{equation}\label{Holderregular}
d(x,F^{-1}(y))\leq \tau [d(y,F(x))]^{q}\quad\mbox{for all} \quad (x,y)\in U\times V.
\end{equation}
 The infimum of all moduli $\tau$ satisfying (\ref{Holderregular})  is denoted by $\reg^{q} F(\bar x,\bar y),$ i.e.,
$$\mbox{reg}_{q} F(\bar x,\bar y)=\inf\{\tau>0: \exists\delta>0 \;s.t.\; d(x,F^{-1}(y))\leq \tau (d( y,F(x)))^{q}\quad\mbox{for all} \quad (x,y)\in B(\bar x,\delta)\times B(\bar y,\delta)\}.$$
Fixing $y=\bar y$ in the above definition, gives  the concept of {\it q-H\"older metric subregularity} of the set-valued mapping $F$ at $(\bar x,\bar y)$.
\vskip 2mm
In the  present paper, we are interested in a \textit{directional version} of H\"{o}lder metric regularity, defined as follows.
\begin{definition}\label{Def-Direc}
Let $X, Y$ be normed linear spaces. Let   a real
$\gamma\in]0,1]$ and   $(u,v)\in X\times Y$ be given. A multifunction  $F$ is
said to be (directionally) metrically $\gamma$-regular at   $(\bar x,\bar
y)\in \gph F$ in  the direction $(u,v)$ with a modulus  $\tau>0$ iff there
exists  $\delta,\varepsilon,\eta>0$ such that
\begin{equation}\label{DHreg}
d(x,F^{-1}(y))\leq \tau [d(y,F(x))]^{\gamma}
\end{equation}
for all $(x,y)\in B((\bar x,\bar y),\delta)$ with $(x,y)\in (\bar x,\bar y)+\cone{B}((u,v),\varepsilon);$ $d(y,F(x))\le \eta\Vert (x,y)-(\bar x,\bar y)\Vert^{1/\gamma}.$  Here $\cone B((u,v),\varepsilon)$ stands for the conic hull of ${B}((u,v),\varepsilon)$, i.e., $\cone {B}((u,v),\varepsilon)=\cup_{\lambda
\geq 0} \lambda{B}((u,v),\varepsilon).$
\end{definition}
\begin{remark}
In Definition \ref{Def-Direc},  the triple $\delta, \varepsilon, \eta>0$ may be replaced by  just a single positive number.
\end{remark}
If (\ref{DHreg}) is required to be verified only at $y=\bar y,$ and $x\in B(\bar x,\delta)$ with $x\in \bar x+\cone{B}(u,\varepsilon),$ we say that $F$ is {\it directionally H\"{o}lder metrically subregular}   at $(\bar x,\bar y)$ in the direction $u.$ When $\gamma=1,$ one refers to the {\it (Lipschitz) directional metric regularity}, as equivalently introduced by Gfrerer  \cite{Gfrerer-SVA}.
\vskip 0.2cm
Since in the definition above, the gauge condition $d(y,F(x))\le \eta\Vert (x,y)-(\bar x,\bar y)\Vert^{1/\gamma}$ is added, when $(u,v)=(0,0),$ the version of  H\"{o}lder/Lipschitz metric regularity in Definition \ref{Def-Direc} is even weaker than  the usual ones defined by (\ref{Holderregular}). For example, obviously, the function $f(x)=x^2,$ $x\in\R$ is H\"{o}lder metrically regular of order $1/2$ at $(0,0)$ in the  direction $(0,0)$ in the sense of Definition \ref{Def-Direc},
but is not in the usual sense of (\ref{Holderregular}). The added gauge conditions in concepts of metric regularity are really needed when the usual regularity is not satisfied (see, e.g., Ioffe \cite{I-Hreg}). Let us mention that directional  H\"{o}lder/Lipschitz metric regularity is obviously stronger than H\"{o}lder/Lipschitz metric subregularity.
The main purpose of the present paper is to show that the tools of variational analysis such as  the  f and the concept of  coderivative can be used to efficiently characterize  directional H\"{o}lder/Lipschitz metric regularity. Our aim is to show  that this directional version of H\"{o}lder/Lipschitz metric regularity, although  weaker than  the usual metric regularity, possesses the suitable stability properties which are   lost in the case of metric subregularity.

\section{Slope Characterizations of Directional H\"older Metric Regularity }
 Let  $X$  be a metric space. Let  $f : X \rightarrow \mathbb{R} \cup
\{+\infty\}$  be a given extended-real-valued function.  As usual, $\mbox{dom}\,f := \{x
\in X : f(x) < +\infty\}$ denotes the domain of $f$.
\vskip 0.2cm Recall from  \cite  {RefDMT}, (see also  \cite  {RefAC2}  and the discussion in \cite  {Kr15})  that the \textit{local
 slope}    and the   \textit{nonlocal slope}  (see, e.g.,  \cite  {FabHenKruOut12}) of the
function $f$ at $x\in\mbox{dom}f$, are the quantities  denoted  respectively by $\vert\nabla f\vert (x)$ and $\vert\Gamma f\vert(x)$.
Using the notation $[a]_+$  for  $\max\{a,0\}$, they are defined by
$\vert\nabla f\vert (x)=\vert\Gamma f\vert(x)=0$ if $x$ is a local minimum of $f$ and otherwise by
\begin{equation}\label{local slope}\vert\nabla f\vert (x)=\limsup_{y\to x,\;y\ne x}\frac{f(x)-f(y)}{d(x,y)}\end{equation}
and
\begin{equation}\label{non local slope}\vert\Gamma f\vert(x)=\sup_{y\ne x}\frac{[f(x)-f(y)]_+}{d(x,y)}.\end{equation}
For $x\notin \mbox{dom}f,$  we set  $\vert\nabla f\vert
(x)=\vert\Gamma f\vert(x)=+\infty.$ Obviously, $\vert\nabla f\vert (x)\le\vert\Gamma f\vert(x)$ for all $x\in X$.
\vskip 0.5cm
Let $X, Y$ be normed spaces. If not
specified otherwise, we assume that  the norm on $X\times Y$ is defined by
$$ \Vert (x,y)\Vert=\Vert x\Vert +\Vert y\Vert,\; (x,y)\in X\times Y.$$ For a  \textit{closed multifunction}
$F:X\rightrightarrows Y$, (i.e., when  the graph of $F$ is closed in $X\times Y$),
the \textit{lower semicontinuous envelope} of the distance function $(x,y)\rightarrow d(y,F(x))$ is defined,  for a  given $(x,y)\in X\times Y$,   by
$$\varphi(x,y):=\liminf_{(u,v)\to (x,y)}d(v,F(u))=\liminf_{u\to x}d(y,F(u)).$$
In what follows,  for $u\in X$,  we shall use the notation   $x\underset {u} {\rightarrow}\bar x$ to mean
$$\begin{array}{lll}& x\to \bar x&\text{if} \;u=0,
 \\
&\left\|\frac{x-\bar x}{\Vert x-\bar x \Vert}-\frac{u}{\Vert u\Vert}\right\|\to 0&
\\&x\to \bar x& \; \text{if}\; u\ne 0.
\end{array}$$
Obviously, given a sequence $\{ x_n\}\subseteq X, u\in X$   the two facts are equivalent:
\begin{enumerate}
\item [\textbf{(C}$_1$)]:
$\{x_n\}\underset{u}\rightarrow \bar x$;
\item [\textbf{(C}$_2$)]:
 $\{x_n\}\to \bar x$ and  there is a sequence of nonnegative reals  $\{\delta_n\}\to 0$ such that
$$ x_n\in \bar x + \cone B(u,\delta_n), \forall n\in \N.$$
\end{enumerate}
We need  the following series of useful  lemmas whose proofs are straightforward.
\begin{lemma}\label{Tron}
Let $X$ be a Banach space and $Y$ be a normed space. Suppose a  closed multifunction
$F:X\rightrightarrows Y$  and    a point $(\bar x,\bar y)\in \gph F $ are  given.   Given $(u,v)\in X\times Y$, and $\gamma\in ]0,1]$, then $F$ is metrically $\gamma-$regular at $(\bar x,\bar y)$ in  the direction $(u,v)$ with modulus $\tau>0$,  if and only if,
there exist real numbers  $\tau, \delta>0$ such that
$$d(x,F^{-1}(y))\le \tau\varphi^\gamma(x,y)$$
$\mbox{for all}\; \;(x,y)\in B((\bar x,\bar y),\delta)\cap ((\bar x,\bar y)+\cone B((u,v),\delta)) \;\text{with}\; \; d(y,F(x))\le \delta\Vert (x,y)-(\bar
x,\bar y)\Vert^{\frac{1}{\gamma}}.$
\end{lemma}
\vskip 0.5cm
\begin{lemma}\label{Direc-conver} Let $u\in X$ and $\bar x\in X$  be given as well as a sequence $\{x_n\}$ such that $x_n\underset{u}\rightarrow \bar x.$ Then, for any sequence  $\{\delta_n\} \downarrow  0$ of  nonnegative reals and any sequence $\{z_n\}\subseteq X$ with
\begin{equation}\label{genevieve}
 \|z_n-x_n\| \le \delta_n\|x_n-\bar x\|,\; n\in\N,\end{equation}
one has $z_n\underset{u}\rightarrow\bar x.$
\end{lemma}
\vskip 0.2cm
\begin{proof}  It suffices to prove the result when $\Vert u\Vert=1.$ If $x_n=\bar x,$ then from (\ref{genevieve}) we have  $z_n=\bar x$ and we are done.
 Otherwise, one has
$$ (1-\delta_n)\Vert x_n-\bar x\Vert\le \Vert z_n-\bar x\Vert,$$
and since
$$ \begin{array}{ll}
\Vert z_n-\bar x-u\Vert z_n-\bar x\Vert\Vert&\le 2\Vert z_n-x_n\Vert +\Vert x_n-\bar x-u\Vert x_n-\bar x\Vert\Vert \\
&\le 2\delta_n\Vert x_n-\bar x\Vert+\Vert x_n-\bar x-u\Vert x_n-\bar x\Vert\Vert,
\end{array}$$
it follows that
$$\left\Vert\frac{z_n-\bar x}{\Vert z_n-\bar x\Vert}-u\right\Vert=\frac{\Vert z_n-\bar x-u\Vert z_n-\bar x\Vert\Vert}{\Vert z_n-\bar x\Vert}\\
\le(1-\delta_n)^{-1}\left(2\delta_n +\left\Vert\frac{x_n-\bar x}{\Vert x_n-\bar x\Vert}-u\right\Vert\right).$$
As $\left\Vert\frac{x_n-\bar x}{\Vert x_n-\bar x\Vert}-u\right\Vert\to 0$ and $\{\delta_n\} \to 0,$ one obtains
$$\left\Vert\frac{z_n-\bar x}{\Vert z_n-\bar x\Vert}-u\right\Vert\to 0$$ as $n$ tends to infinity. Moreover, it is easy to see that $z_n\to\bar x$. So,  one has that $z_n\underset{u}\rightarrow\bar x$.  \hfill{$\Box$} \end{proof}
\vskip 0.5cm
\begin{theorem}\label{theo1}
Let $X$ be a Banach space and $Y$ be a normed space. Suppose  a closed multifunction
$F:X\rightrightarrows Y$ and   a point $(\bar x,\bar y)\in \gph F $ are given. Given $(u,v)\in X\times Y$, and $\gamma\in ]0,1]$, $F$ is directional  metrically $\gamma$-regular at $(\bar x,\bar y)$ in the direction $(u,v)$,  if and only if,
\begin{equation} \label{EvSNTDK1}
\underset{\underset{\frac{\varphi^\gamma(x,y)}{\Vert (x,y)-(\bar x,\bar y)\Vert}\to 0}{ (x,y)\underset {(u,v)} {\rightarrow}(\bar x,\bar y),\;\varphi(x,y)>0}}{\liminf} \vert \Gamma \varphi^\gamma(\cdot,y)\vert (x)>0.
\end{equation}
\end{theorem}
\vskip 0.2cm
\begin{proof}
For the sufficiency, assume that  (\ref{EvSNTDK1}) holds and assume on the contrary that $F$ fails to be directional metrically $\gamma$-regular in the  direction $(u,v)$. Then for every $n\in\N$,  there exists $(x_n,y_n)$ with
$$
 0<\Vert x_n-\bar x\Vert+\Vert y_n-\bar y\Vert<\frac{1}{n};\;d(y_n,F(x_n))^\gamma\le \frac{1}{n^2} \Vert (x_n,y_n)-(\bar x,\bar y)\Vert;$$
$$ (x_n,y_n)\underset{(u,v)}\rightarrow (\bar x,\bar y)$$
\text{and such that}
$$d(x_n, F^{-1}(y_n))>n^2 \left[d(y_n,F(x_n))\right]^\gamma(\geq n^2 \varphi^\gamma(x_n,y_n)).$$
From the relation $d(y_n,F(x_n))^\gamma\le \frac{1}{n^2} \Vert (x_n,y_n)-(\bar x,\bar y)\Vert,$ it follows that
\begin{equation}\label{nantes}
\varphi^\gamma(x_n,y_n)\le \frac{1}{n^2} \Vert (x_n,y_n)-(\bar x,\bar y)\Vert.
\end{equation}
Applying  the Ekeland variational principle to  the lower semicontinuous function  $\varphi^\gamma(\cdot,y_n)$ on the Banach space $X$, one gets a point $z_n$ satisfying the  conditions:
$$\Vert z_n-x_n\Vert \le \frac{1}{n}\Vert (x_n,y_n)-(\bar x,\bar y)\Vert,\; \varphi^\gamma(z_n,y_n)\le \varphi^\gamma(x_n,y_n),$$ and
$$\varphi^\gamma(z_n,y_n)\le \varphi^\gamma(x,y_n)+\frac{1}{n}\Vert x-z_n\Vert, \;\forall\; x\in X.$$
Consequently,
 \begin{equation}\label{poussy-bis}
\vert\Gamma\varphi^\gamma(\cdot,y_n)\vert (z_n)\le \frac{1}{n},
\end{equation}
and by
$$ \Vert (x_n,y_n)-(\bar x,\bar y)\Vert\le  \Vert z_n-x_n\Vert+\Vert (z_n,y_n)-(\bar x,\bar y)\Vert\le \frac{1}{n}\Vert (x_n,y_n)-(\bar x,\bar y)\Vert+ \Vert (z_n,y_n)-(\bar x,\bar y) \Vert,$$
one obtains
\begin{equation}\label{poussy}
\Vert (x_n,y_n)-(\bar x,\bar y)\Vert\le \frac{n}{n-1}\Vert (z_n,y_n)-(\bar x,\bar y)\Vert.
\end{equation}
 Hence, combining (\ref{nantes}) and (\ref{poussy}) we obtain:
$$\varphi^\gamma(z_n,y_n)\le \varphi^\gamma(x_n,y_n)\le \frac{1}{n^2} \Vert (x_n,y_n)-(\bar x,\bar y)\Vert\le \frac{1}{n(n-1)}\Vert (z_n,y_n)-(\bar x,\bar y)\Vert.$$
The latter relation implies that
   $$\lim_{n\to\infty}\frac{\varphi^\gamma(z_n,y_n)}{\Vert (z_n,y_n)- (\bar x,\bar y)\Vert}=0.$$
Moreover, invoking Lemma \ref{Direc-conver},  relations $ (x_n,y_n) \underset{(u,v)}\rightarrow (\bar x,\bar y)$ and
$\Vert z_n -x_n\Vert \le 1/n\Vert (x_n,y_n)-(\bar x,\bar y)\Vert$  imply that \\ $(z_n,y_n)\underset{(u,v)}\rightarrow (\bar x,\bar y).$
Hence, by (\ref{EvSNTDK1}), $z_n\in F^{-1}(y_n)$ for $n$   large, say, $n\ge n_0.$
\vskip 0.1cm
\noindent For $n\ge n_0,$ as
$$ \varphi^\gamma(x_n,y_n)<\frac{1}{n^2}d(x_n,F^{-1}(y_n)),$$
 applying  again  the Ekeland variational principle to the lower semicontinuous $\varphi^\gamma(\cdot,y_n)$, one gets a point $w_n\in X$ satisfying the conditions:
$$\Vert w_n-x_n\Vert <\frac{1}{n}d(x_n,F^{-1}(y_n)),\; \varphi^\gamma(w_n,y_n)\le \varphi^\gamma(x_n,y_n)$$
and,
$$\varphi^\gamma(w_n,y_n)\le \varphi^\gamma(x,y_n)+\frac{1}{n}\Vert x-w_n\Vert, \;\forall\; x\in X.$$
We deduce that $$\vert\Gamma\varphi^\gamma (\cdot,y_n)\vert (w_n)\le \frac{1}{n}.$$ Moreover  we claim  that
\begin{equation}\label{dominicain}
\frac{\varphi^\gamma(w_n,y_n)}{\Vert (w_n,y_n)-(\bar x,\bar y)\Vert}\le \frac{1}{n^2}+\frac{1}{n^2(n^2-1)}.
\end{equation}
Indeed, as
\begin{align*}
\Vert (x_n,y_n)-(\bar x,\bar y)\Vert\le\\
&\le \Vert (x_n,y_n)-(w_n, y_n)\Vert+\Vert (w_n,y_n)-(\bar x,\bar y)\Vert\\
&= \Vert x_n-w_n\Vert+\Vert (w_n,y_n)-(\bar x,\bar y)\Vert\\
&\le \frac{1}{n}d(x_n,F^{-1}(y_n))+\Vert (w_n,y_n)-(\bar x,\bar y)\Vert\\
&\le \frac{1}{n}\Vert x_n-z_n\Vert+\Vert (w_n,y_n)-(\bar x,\bar y)\Vert\\
&\le  \frac{1}{n^2}\Vert (x_n,y_n)-(\bar x,\bar y)\Vert+\Vert (w_n,y_n)-(\bar x,\bar y)\Vert,
\end{align*}
we have
$$\Vert (x_n,y_n)-(\bar x,\bar y)\Vert\le \frac{n^2}{n^2-1}\Vert (w_n,y_n)-(\bar x,\bar y)\Vert,$$
and
\begin{align*}
\varphi^\gamma(w_n,y_n)&\le \varphi^\gamma(x_n,y_n)\le \frac{1}{n^2}\Vert (x_n,y_n)-(\bar x,\bar y)\Vert\\
 &\le  \frac{1}{n^2}\Vert (w_n,y_n)-(\bar x,\bar y)\Vert+\frac{1}{n^2}\Vert x_n-w_n\Vert\\
& \le \frac{1}{n^2}\Vert (w_n,y_n)-(\bar x,\bar y)\Vert+\frac{1}{n^3}d(x_n,F^{-1}(y_n))\le \frac{1}{n^2}\Vert (w_n,y_n)-(\bar x,\bar y)\Vert+\frac{1}{n^3}\Vert x_n-z_n\Vert\\
&\le \frac{1}{n^2}\Vert (w_n,y_n)-(\bar x,\bar y)\Vert+\frac{1}{n^4}\Vert (x_n,y_n)-(\bar x,\bar y)\Vert.
\end{align*}
Therefore,
$$\varphi^\gamma(w_n,y_n)\le \left(\frac{1}{n^2}+\frac{1}{n^2(n^2-1)}\right)\Vert (w_n,y_n)-(\bar x,\bar y)\Vert,$$
and (\ref{dominicain}) is established.
By virtue  of Lemma \ref{Direc-conver}, since $(x_n,y_n)\underset {(u,v)}\rightarrow(\bar x,\bar y)$ and $\Vert w_n-x_n\Vert\le \frac{1}{n^2}\Vert (x_n,y_n)-(\bar x,\bar y)\Vert,$ one has $(w_n,y_n)\underset{(u,v)}\rightarrow(\bar x,\bar y).$
\vskip 0.1cm
In conclusion, we have obtained a sequence $\{w_n\}$  which satisfies
$$w_n\notin F^{-1}(y_n);\; (w_n,y_n)\underset{(u,v)}\rightarrow(\bar x,\bar y);\; \frac{\varphi^\gamma(w_n,y_n)}{\Vert (w_n,y_n)-(\bar x,\bar y)\Vert}\to 0\quad\mbox{and}\; \vert\Gamma\varphi^\gamma (\cdot,y_n)\vert (w_n)\le \frac{1}{n}.$$
Hence, condition (\ref{EvSNTDK1}) is violated, and the sufficiency is proved.\\
\vskip 0.2cm
For the necessary part, suppose that there exist {reals}  $\tau, \delta>0$ such that
$$d(x,F^{-1}(y))\le \tau [d(y, F(x))]^\gamma$$
$\text{for all}\; \;(x,y)\in B((\bar x,\bar y),\delta)\cap ((\bar x,\bar y)+\cone B((u,v),\delta)) \;\text{with}\; \; d(y,F(x))\le \delta\Vert (x,y)-(\bar
x,\bar y)\Vert^{\frac{1}{\gamma}}.$\\
According to  Lemma \ref{Tron},
$$d(x,F^{-1}(y))\le \tau\varphi^\gamma(x,y)\;\forall (x,y)\in B((\bar x,\bar y),\delta)\cap ((\bar x,\bar y)+\cone B((u,v),\delta))\;\mbox{with}\;0<\frac{\varphi^\gamma(x,y)}{\Vert (x,y)-(\bar x,\bar y)\Vert}\le \delta.$$
Let $(x,y)\in B((\bar x,\bar y),\delta)\cap ((\bar x,\bar y)+\cone B((u,v),\delta))$ with $(x,y)\ne (\bar x,\bar y)$ and $0<\frac{\varphi^\gamma(x,y)}{\Vert (x,y)-(\bar x,\bar y)\Vert}\le \delta.$ Then, for every $\varepsilon>0,$ there exists an element  $z\in F^{-1}(y)$ such that $$\Vert x-z\Vert \le (\tau+\varepsilon) \varphi^\gamma(x,y)= (\tau+\varepsilon) [\varphi^\gamma(x,y)-\varphi^\gamma(z,y)] .$$ Consequently, $$\vert \Gamma \varphi^\gamma(\cdot,y)\vert(x)\geq \frac{1}{(\tau+\varepsilon)}.$$
As $\varepsilon>0$ is arbitrary, one obtains
 $$
\underset{\underset{\frac{\varphi^\gamma(x,y)}{\Vert x-\bar x\Vert}\to 0}{ (x,y)\underset {(u,v)} {\rightarrow}(\bar x,\bar y),\;\varphi(x,y)>0}}{\liminf} \vert \Gamma \varphi^\gamma(\cdot,y)\vert (x)\geq \frac{1}{\tau}>0,
$$
completing  the proof.\hfill{$\Box$} \end{proof}
\vskip 0.2cm

The theorem above  yields the following local  slope  characterization of the directional metric $\gamma-$regularity .
\begin{theorem}\label{Slope-Chac}
Let $X$ be a Banach space and $Y$ be a normed space. Suppose a closed  multifunction
$F:X\rightrightarrows Y$  and    a point $(\bar x,\bar y)\in X\times Y$  are given such that $\bar y\in F(\bar x).$  Let   $(u,v)\in X\times Y$, and $\gamma\in (0,1]$ be fixed.
If
\begin{equation} \label{EvSNT!}
\underset{\underset{\frac{\varphi^\gamma(x,y)}{\Vert (x,y)-(\bar x,\bar y)\Vert}\to 0}{ (x,y)\underset {(u,v)} {\rightarrow}(\bar x,\bar y),\;\varphi(x,y)>0}}{\liminf} \vert \nabla \varphi^\gamma(\cdot,y)\vert (x)>0,
\end{equation}
then there exist reals  $\tau, \delta>0$ such that
$$d(x,F^{-1}(y))\le \tau [d(y, F(x))]^\gamma$$
$\text{for all}\; \;(x,y)\in B((\bar x,\bar y),\delta)\cap ((\bar x,\bar y)+\cone B((u,v),\delta)) \;\text{with}\; d(y,F(x))\le \delta\Vert (x,y)-(\bar
x,\bar y)\Vert^{\frac{1}{\gamma}}.$\\
That is, $F$ is directionally  metrically $\gamma$-regular at $(\bar x,\bar y)$ in the  direction $(u,v)$ with modulus $\tau$.
\end{theorem}
\begin{remark}
{\rm Condition (\ref{EvSNT!}) of   Theorem \ref{Slope-Chac} fails to be a necessary condition when $\gamma\in ]0,1[$. To see this, let us consider the mapping $F:\R^2\to\R$ defined by
$$ F(x) = (x_1-x_2)^3,\quad x=(x_1,x_2)\in\R^2.$$ Here, $\R^2$ is equipped with the Euclidean norm.
For $x=(x_1,x_2)\in\R^2$ and $y\in\R,$ $F^{-1}(y)=\{(t+\sqrt[3]{y},t):\;\;t\in\R\}.$ Therefore, $d(x,F^{-1}(y))=|x_1-x_2-\sqrt[3]{y}|/\sqrt{2}.$
Noticing that  from the inequality  $0\leq 3(a+b)^2$, we deduce that \\$(a-b)^2\leq 4(a^2+ab+b^2)$ and therefore that
$ |a-b|^3\le 4|a^3-b^3|,$  for all  $a,b\in \R$. Using the last inequality yields
$$ d(x,F^{-1}(y))=|x_1-x_2-\sqrt[3]{y}|/\sqrt{2}\le 2^{1/6} |(x_1-x_2)^3-y|^{1/3}=2^{1/6}|y-F(x)|^{1/3}\;\;\mbox{for all}\; x\in\R^2,\, y\in\R. $$
Consequently, $F$ is metrically $1/3-$regular at $(0,0)$ (in  the direction $(0,0)$). However, for any $x=(x_1,x_1)\in\R^2$ and $y\in\R$ with $y\not=0,$ one has
$\vert \nabla \varphi^{1/3}(\cdot,y)\vert (x)=0,$ where, $\varphi(x,y)=|y-F(x)|.$
\vskip 0.2cm
 As stated in the next theorem, when $\gamma=1$, then condition  (\ref{EvSNT!}) becomes  a necessary  condition.
 }
\end{remark}
\vskip 0.2cm
\begin{theorem}\label{hebdo}
Let $X$ be a Banach space and $Y$ be a normed space. Suppose  a closed  multifunction
$F:X\rightrightarrows Y$ and   a point    $(\bar x,\bar y)\in X\times Y$  are given such that $\bar y\in F(\bar x).$   Let us fix $(u,v)\in X\times Y$.
Then, the following are equivalent:
\begin{enumerate}
\item[(i)]
 \begin{equation}\label{EvSNT1}
\underset{\underset{\frac{\varphi(x,y)}{\Vert (x,y)-(\bar x,\bar y)\Vert}\to 0}{ (x,y)\underset {(u,v)} {\rightarrow}(\bar x,\bar y),\;\varphi(x,y)>0}}{\liminf} \vert \nabla \varphi(\cdot,y)\vert (x)>0,
\end{equation}
\item [(ii)]there exist {reals}  $\tau>0$ and $\delta>0$ such that
$$d(x,F^{-1}(y))\le \tau d(y, F(x))$$
$\text{for all}\; \;(x,y)\in B((\bar x,\bar y),\delta)\cap ((\bar x,\bar y)+\cone B((u,v),\delta)) \;\text{with}\; d(y,F(x))\le \delta\Vert (x,y)-(\bar
x,\bar y)\Vert.$\\
That is $F$ is directionally  metrically regular at $(\bar x,\bar y)$ in  the direction $(u,v)$ with modulus $\tau$.
\end{enumerate}
\end{theorem}
\begin{proof} According to Theorem \ref{Slope-Chac}, condition (\ref{EvSNT1}) is sufficient to obtain (ii).
Conversely, suppose  that there exist $\tau>0$ and $\delta\in]0,1[$ such that
\begin{equation}\label{TT}
d(x,F^{-1}(y))\le \tau d(y, F(x))
\end{equation}
$\text{for all}\; \;(x,y)\in B((\bar x,\bar y),2\delta)\cap ((\bar x,\bar y)+\cone B((u,v),2\delta)) \;\text{with}\; d(y,F(x))\le 2\delta\Vert (x,y)-(\bar
x,\bar y)\Vert.$
\vskip 2mm
Let $(x,y)\in B((\bar x,\bar y),\delta)\cap ((\bar x,\bar y)+\cone B((u,v),\delta))$ be such that
 \begin{equation}\label{pierrot}
 0<d(y,F(x))\le \frac{\delta^2}{4}\Vert (x,y)-(\bar
x,\bar y)\Vert.
\end{equation}
Since $\displaystyle\varphi(x,y)=\sup_{\varepsilon>0} \inf_{w\in B(x,\varepsilon)}d(y,F(w))=\liminf_{w\to x}d(y, F(w))$, we have
$$\frac{\varphi(x,y)}{\Vert (x,y)-(\bar x,\bar y)\Vert}\le\frac{\delta^2}{4}$$
and we can write $\displaystyle\varphi(x,y)=\lim_{n\to +\infty} d(y, F(u_n)$) for some  sequence
 $\{u_n\}$ in $X$  converging to $x$.
Without loss of generality, we can suppose that $d(y,F(u_n))>(1-\frac{1}{n^2})\varphi(x,y),$ and
\begin{equation}\label{Tr7}
\Vert u_n-x\Vert<\frac{1}{n}\varphi(x,y)\le \frac{\delta^2\Vert (x,y)-(\bar
x,\bar y)\Vert}{4n} \leq \frac{\delta^3}{4n}< \frac{\delta}{n}
\end{equation}
and
\begin{equation}\label{valor}
d(y,F(u_n))<\left(1+\frac{1}{n}\right)\varphi(x,y).
\end{equation}
Note also that for every $n\in\N$, there exists $y_n\in F(u_n)$ such that
\begin{equation}\label{couzeix}
d(y,F(u_n))\le \Vert y-y_n\Vert<\left(1+\frac{1}{n}\right)d(y,F(u_n)).\end{equation}
Set  $z_n:=\frac{1+n^\frac{1}{2}}{1+n}y+\frac{n(1-n^{-\frac{1}{2}})}{1+n}y_n.$
\vskip 2mm
\begin{claim}\begin{equation}\label{star}\Vert y-z_n\Vert<\left(1-\frac{1}{n}\right)\varphi(x,y) \;\text{for large\;\; n.}\end{equation}
\end{claim}
Indeed,
\begin{equation}\label{cate}\Vert y_n-z_n\Vert \leq \frac{1+n^{1/2}}{1+n} \Vert y-y_n\Vert,\end{equation}
and
\begin{align*}
\Vert y-z_n\Vert=\frac{\ n(1-n^{-\frac{1}{2}})}{1+n}\Vert y-y_n\Vert &<\frac{n(1-n^{-\frac{1}{2}})}{1+n}\left(1+\frac{1}{n}\right)d(y,F(u_n))\nonumber\\
&<\left(1-n^{-\frac{1}{2}}\right)d(y,F(u_n))<\left(1-n^{-\frac{1}{2}}\right)\left(1+\frac{1}{n}\right)\varphi(x,y)\nonumber\\
&\le \left(1-n^{-\frac{1}{2}}\right)\left(1+n^{-\frac{1}{2}}\right)\varphi(x,y).
\end{align*}
Hence, $\Vert y-z_n\Vert<  (1-\frac{1}{n})\varphi(x,y),$ as claimed.
\vskip 2mm
\begin{claim}
 $z_n\notin F(u_n) \; \text{for large}\;\; n.$
\end{claim}
Indeed, using (\ref{star}), if  we suppose that   $z_n\in F(u_n)$ (for $n\geq n_0$), we deduce that
$$(1-\frac{1}{n^2})\varphi(x,y)< d(y,F(u_n))\le \Vert y-z_n\Vert<\left(1-\frac{1}{n}\right)\varphi(x,y),$$
 we have a contradiction.
Hence,  as claimed, for $n$ large,  $z_n\notin F(u_n)$.%
\begin{claim}$$(\star\star)\quad (u_n,z_n)\in(\bar x,\bar y)+\cone B((u,v),2\delta) \; \text{for large}\;\; n.$$
\end{claim}
As $(x,y)\in (\bar x,\bar y)+\cone B( (u,v),\delta),$ there is some $\lambda>0$ such that $\frac{(x,y)-(\bar x,\bar y)}{\lambda}\in B((u,v),\delta).$ Then,
$$ \lambda\ge \frac{\Vert (x,y)-(\bar x,\bar y)\Vert}{\Vert (u,v)\Vert+\delta}.$$
Let us observe that
$$\frac{\Vert z_n-y\Vert}{\lambda}\underset {by (\ref{star})}{<} \frac{(1-1/n)}{\lambda}\varphi(x,y)<\frac{(1-1/n)\delta^2}{4n\lambda}\Vert (x,y)-(\bar x,\bar y)\Vert\leq \frac{(1-\frac{1}{n})\delta^2}{4n}(\Vert (u,v)\Vert +\delta)\leq \frac{(1-\frac{1}{n})\delta^2}{4}\leq \frac{\delta}{2}.$$
 Therefore, one has
 $$\begin{array}{ll}\left\| \frac{(u_n,z_n)-(\bar x,\bar y)}{\lambda}-(u,v)\right\| &
 \\\leq \left\| \frac{ (u_n, z_n)-(x,y) }{\lambda}\right\| +\left\| \frac{(x,y)-(\bar x,\bar y)}{\lambda}-(u,v)\right\|
  \\ \leq \frac{\Vert u_n-x\Vert +\Vert z_n-y\Vert}{\lambda}+\delta&\\
 \le \frac{\delta}{\lambda n}+\frac{\delta}{2} +\delta< 2\delta.
\end{array}$$
This yields, ($\star\star$).
\begin{claim} $$(\star\star\star)\quad d(z_n,F(u_n))<2\delta^2\Vert (u_n,z_n)-(\bar x,\bar y)\Vert \;\; \text{for large}\;\; n.$$
\end{claim}
We know that  $$(\clubsuit)\quad d(z_n,F(u_n))<\frac{1+n^{1/2}}{1+n}(1+\frac{1}{n})^2\varphi(x,y)\le \frac{1+n^{1/2}}{1+n}(1+\frac{1}{n})^2\frac{\delta^2}{4}\|(x,y)-(\bar x,\bar y)\|.$$
 We have:
$$\|(x,y)-(\bar x,\bar y)\|\le \|(u_n,z_n)-( x,y)\|+\|(u_n,z_n)-(\bar x,\bar y)\|,$$
and
$$\|(u_n,z_n)-( x,y)\|=\Vert u_n-x\Vert+\Vert z_n-y\Vert<\frac{\delta^2}{4n}\Vert (x,y)-(\bar x,\bar y)\Vert+(1-1/n)\frac{\delta^2}{4n}\Vert (x,y)-(\bar x,\bar y)\Vert=\frac{\delta^2}{4n}\Vert (x,y)-(\bar x,\bar y)\Vert,$$
 it follows
 $$(\heartsuit)\quad \|(x,y)-(\bar x,\bar y)\|<\frac{1}{(1-\frac{\delta^2}{4n})}\|(u_n,z_n)-( \bar x,\bar y)\|.$$
Hence combining $(\clubsuit)$ and $(\heartsuit)$ we derive
 $$d(z_n,F(u_n))<\frac{1+n^{1/2}}{1+n}(1+\frac{1}{n})^2\frac{\delta^2}{4}\|(x,y)-(\bar x,\bar y)\|< \frac{1+n^{1/2}}{1+n}(1+\frac{1}{n})^2\frac{\delta^2}{4(1-\frac{\delta^2}{4n})}\|(u_n,z_n)-( \bar x,\bar y)\|.$$
Observing that the quantity $\frac{1+n^{1/2}}{1+n}\frac{(1+\frac{1}{n})^2}{(1-\frac{\delta^2}{4n})}$ tends to $0$ as $n$ tends to $+\infty$,  we obtain  $(\star\star\star)$, as desired.
\begin{claim}
\begin{equation}\label{vals}
\Vert (u_n,z_n)-(\bar x,\bar y)\Vert \leq 2\delta \;\text{for large}\;\; n.
\end{equation}
\end{claim}
\begin{align*}
\Vert (u_n,z_n)-(\bar x,\bar y)\Vert & \leq \Vert (u_n,z_n)-(x,y)\Vert +\Vert (x,y)-(\bar x,\bar y)\Vert\\
&\leq \frac{\delta}{n} +(1-\frac{1}{n})\varphi (x,y)+\Vert ((x,y)-(\bar x,\bar y)\Vert\\
&\leq   \frac{\delta}{n} +(2-\frac{1}{n})\frac{\delta^2}{4} \Vert (x,y)-(\bar x,\bar y)\Vert
\\
&\leq 2\delta.
\end{align*}
Combining relations ($\ref{vals}$), ($\star\star$) and ($\star\star\star$) we see that  the point $(u_n,z_n)$ verifies  (\ref{TT}). Hence by assumption we have
$$d(u_n, F^{-1}(z_n))\leq \tau d(z_n,F(u_n).$$
Next, select $\tilde{x}_n\in F^{-1}(z_n)$ such that
\begin{align}
\Vert \tilde{x}_n-u_n\Vert&\le (1+n^{-\frac{1}{2}})d(u_n, F^{-1}(z_n))\nonumber\\
&\le\tau (1+n^{-\frac{1}{2}})d(z_n,F(u_n)\nonumber)\\
&\le\tau  (1+n^{-\frac{1}{2}})\Vert z_n-y_n\Vert\nonumber\\
&=\tau  (1+n^{-\frac{1}{2}})\frac{(1+n^\frac{1}{2})}{1+n}\Vert y-y_n\Vert\nonumber\\
&=\tau \frac{2+n^\frac{1}{2}+n^{-\frac{1}{2}}}{1+n}\Vert y-y_n\Vert\label{marionelie}.
\end{align}
Consequently,
$\Vert \tilde{x}_n-u_n\Vert <\tau \frac{ 2+n^1/2+n^{-1/2}}{1+n}\Big(1+\frac{1}{n}\Big)^2 \varphi (x,y)$ and therefore, $\lim_{n\to\infty}\Vert \tilde{x}_n-x\Vert=0$. Next,  for large $n$, we have the following estimate:
\begin{align*}
\varphi(x,y)-\varphi(\tilde{x}_n,y)&\underset{by (\ref{valor})}>\frac{n}{n+1}\Big(d(y,F(u_n))-d(y,F(\tilde{x}_n))\Big)\\
&\underset{by (\ref{couzeix})}>\frac{n^2}{(n+1)^2}\Vert y-y_n\Vert-\Vert y-z_n\Vert,
\end{align*}
and by the definition of $z_n$, we derive
$$\varphi(x,y)-\varphi(\tilde{x}_n,y)>\frac{n^{3/2}-n+n^{1/2}}{(n+1)^2}\Vert y-y_n\Vert.$$
Thus, using also (\ref{Tr7}), we obtain
\begin{align*}
\frac{\varphi(x,y)-\varphi(\tilde{x}_n,y)}{\Vert \tilde{x}_n-x\Vert}
&\geq \frac{\varphi(x,y)-\varphi(\tilde{x}_n,y)}{\Vert u_n-x\Vert+\Vert\tilde{x}_n-u_n\Vert}\\
&\underset{by (\ref{marionelie})}{>} \frac{(n^{3/2}-n+n^{1/2})(n+1)^{-2}\Vert y-y_n\Vert}{\delta n^{-1}+\tau(2+n^\frac{1}{2}+n^{-\frac{1}{2}})(n+1)^{-1}\Vert y-y_n\Vert }\\
&=\frac{n^{3/2}-n+n^{1/2}}{(2+n^\frac{1}{2}+n^{-\frac{1}{2}})(n+1)}\frac{1}{\tau+\frac{\delta (n+1)}{n(2+n^{1/2}+n^{-\frac{1}{2}})}\Vert y-y_n\Vert^{-1}}.
\end{align*}
Since $\frac{n^{3/2}-n+n^{1/2}}{(2+n^\frac{1}{2}+n^{-\frac{1}{2}})(n+1)}\to 1$, $\lim_{n\to\infty}\Vert y-y_n\Vert=\varphi(x,y)>0$
and
$$\frac{\delta (n+1)}{n(2+n^\frac{1}{2}+n^{-\frac{1}{2}})}\cdot \frac{1}{\Vert y-y_n\Vert}\to 0,$$ we deduce that
$$\vert\nabla \varphi(\cdot,y)\vert(x)= \limsup_{u\to x} \frac{\varphi (x,y)-\varphi (u,y)}{\Vert x-u\Vert}\geq \limsup_{n\to\infty}\frac{\varphi(x,y)-\varphi(\tilde{x}_n,y)}{\Vert \tilde{x}_n-x\Vert}\geq\frac{1}{\tau},$$
establishing the proof.\hfill{$\Box$}\end{proof}
\section{Directional Metric Regularity under Perturbations}
In \cite  {RefDoLR, RefDoL, RefIo2, SiamNT}, it was established that  metric regularity is stable under Lipschitz (single-valued or set-valued) perturbations with a  suitable Lipschitz modulus. We shall show that directional metric regularity (with $\gamma=1$) is also stable under suitable Lipschitz perturbations. Recall that a mapping  $g:X\to Y$ between  normed spaces is Hadamard differentiable at $\bar x\in X$ with respect to the direction $u\in X,$ if the following limit exists:
$$ \lim_{\underset{w\to u}{t\downarrow 0}}\frac{g(\bar x+tw)-g(\bar x)}{t}=Dg(\bar x)(u).$$
Obviously, if $g$ is locally Lipschitz around $\bar x,$ then $g$ is Hadamard  differentiable at $\bar x$ with respect to the direction $0$ and $Dg(\bar x)(0)=0.$
\vskip 0.5cm
\begin{theorem}\label{Pertu}
Let $X$ be a Banach space, $Y$ be a normed space. For  a given $(u,v)\in X\times Y,$ suppose $F: X\rightrightarrows Y$ is a set-valued mapping which is directionally metrically regular at $(\bar x,\bar y)\in\gph F$
in direction $(u,v)$ with modulus  $\tau>0.$ If  $g: X\rightarrow Y$ is   locally  Lipschitz  around $\bar x$ with a constant $\lambda>0$ satisfying $\lambda\tau<1$ and  if  $g$ is
 Hadamard  differentiable at $\bar x$ with respect to  the direction $u,$  then  the set-valued mapping $F+g$ is directionally metrically regular at $(\bar x, \bar y+g(\bar x))$ in the  direction $(u,v+Dg(\bar x)(u))$.
\end{theorem}
\vskip 0.2cm\begin{proof}  Since $g$ is locally Lipschitz around $\bar x$  with constant $\lambda $ and  $F$ is directionally metrically regular at $(\bar x,\bar y)\in \gph F$
in direction $(u,v)$  with modulus $ \tau >0$, take  $\delta\in (0,\lambda)$ such that
$$\Vert g(x)-g(z)\Vert\le\lambda \Vert x-z\Vert \;\; \forall x, z\in B(\bar x,\delta),$$
and
\begin{equation}\label{DMR-F}
\begin{array}{ll} d(x,F^{-1}(y))\le\tau d(y,F(x))
&\mbox{for all}\;\;(x,y)\in B((\bar x,\bar y),\delta)\cap ((\bar x,\bar y)+\cone B((u,v),\delta)),\\
&\text{with}\; \; d(y,F(x))\le 2\delta\Vert (x,y)-(\bar
x\,\bar y)\Vert.
\end{array}
\end{equation}
Taking into account the Hadamard differentiability of $g$ at $\bar x$ with respect to the direction $u$, pick some  $\varepsilon \in \left]0,\frac{\delta}{2}\right[$ and then  some $\eta\in (0,1)$ such that
\begin{equation}\label{Diff}
\left\Vert \frac{g(\bar x+tw)-g(\bar x)}{t}-Dg(\bar x)(u)\right\Vert<\varepsilon\quad\mbox{for all}\; t\in]0,\eta[,\; w\in B(u,\eta).
\end{equation}
Set
$$ \varphi_{F+g}(x,y)=\liminf_{z\to x}d(y,F(z)+g(z)),\;\; (x,y)\in X\times Y. $$
Take $\rho>0$ such that
\[
\rho< \left\{ \begin{array}{ll}
\min{\left\{\frac{\delta}{\lambda +3}, {\delta-\varepsilon}, \frac{\eta \Vert u\Vert }{1+\eta}, \frac{\delta}{2\tau +1}\right\}}, &\;\;\text{if}\;\; u\ne 0\\
$$\rho<\min{\left\{\frac{\delta}{\lambda +3}, {\delta-\varepsilon}, \frac{\delta}{2\tau +1}\right\}}, &\;\; \text{if}\;\; u=0
\end{array} \right. \]
and  fix
$$ (x,y)\in B( (\bar x,\bar y+g(\bar x)),\rho)\cap \left((\bar x,\bar y+g(\bar x))+\cone B( (u,v+Dg(\bar x)(u)),\rho)\right)$$
with
\begin{equation}\label{NNN}
0<d(y,F(x)+g(x))<\rho\Vert(x,y)-(\bar x,\bar y+g(\bar x))\Vert,
\end{equation} and select  a sequence $\{x_n\}$ converging to $x$  such that
$\varphi_{F+g}(x,y)=\lim_{n\to\infty}d(y,F(x_n)+g(x_n)).$
\vskip 2mm
With the  aim of making the proof clearer, we will establish some claims.
\begin{claim}
\begin{equation}\label{TM$$ DK}
\Vert (x_n,y)-(\bar x,\bar y+g(\bar x))\Vert<\delta, \;\text{ for }\; n \; \text{large}.
\end{equation}
\end{claim}
\begin{claim}
\begin{equation} \label{rel1}(x_n,y)\in (\bar x,\bar y +g(\bar x))+ \cone B( (u,v+Dg(\bar x)(u)),\rho)  \;\; \text{ for }\; n \;\text{large.}\end{equation}
\end{claim}
Indeed, since $(x,y)-(\bar x,\bar y+g(\bar x))\notin \cone B( (u,v+Dg(\bar x)(u)),\rho)\setminus\{(0,0)\}, $  for large $n$, one has $$(x,y)-(\bar x,\bar y+g(\bar x))\in \inte \left(\cone B(
(u,v+Dg(\bar x)(u)),\rho\right.))$$ It follows  that  $(x_n,y) -(\bar x,\bar y +g(\bar x))\in  \cone B( (u,v+Dg(\bar x)(u)),\rho).$
\begin{claim}
\begin{equation} \label{madoumier} 0<d(y,F(x_n)+g(x_n))<\rho\Vert(x_n,y)-(\bar x,\bar y+g(\bar x))\Vert \quad \text{for } \; n \;\text{large}. \end{equation}
\end{claim}
Indeed, reasoning ad absurdum, suppose the existence of a  subsequence still denoted by $\{x_n\}$  such that  $$d(y,F(x_n)+g(x_n))\geq \rho\Vert(x_n,y)-(\bar x,\bar y+g(\bar x))\Vert.$$
Since $F$ has  closed graph then $$0<\varphi_{F+g}(x,y)=\lim _{n\to +\infty}d(y,F(x_n)+g(x_n))\le d(y,F(x)+g(x)), $$
due to the fact that  $\{x_n\}$ converges to $x$
and  thanks to (\ref{NNN}), we obtain a contradiction.

\begin{claim}\label{zee}
For $n$ sufficiently large,
\begin{equation}\label{music}
(x_n,y-g(x_n))\in B((\bar x,\bar y), \delta),
\end{equation}
\begin{equation}\label{music2}
d(y- g(x_n), F(x_n))\leq 2\delta \Vert (x_n,y-g(x_n))-(\bar x,\bar y)\Vert.
\end{equation}
and
\begin{equation}\label{music1}
(x_n,y-g(x_n))\in (\bar x, \bar y) +\cone B((u,v),\delta).
\end{equation}
\end{claim}
We know that
$$\Vert  (x-\bar x, y-\bar y -g(\bar x))\Vert <\rho.$$
For $n$ sufficiently large, one has $\Vert x_n-x\Vert<\delta$ and
\begin{align*}\Vert (x_n-\bar x, y-g(x_n)- \bar y)\Vert& = \Vert x_n -\bar x \Vert + \Vert y - \bar y -g(x_n)\Vert \\&\leq \Vert  x_n - x\Vert +\Vert x-\bar x\Vert +\Vert y -\bar y -g(\bar x)\Vert +\Vert g(\bar x)-g(x_n)\Vert
\\&\leq  \Vert x_n-x\Vert +\Vert (x-\bar x,y-\bar y-g(\bar x))\Vert +\lambda \Vert x_n-x\Vert
\\&=(\lambda +1)\Vert x_n - x\Vert +2\rho
\\&< (\lambda + 3)\rho<\delta \quad (\text{since}\; \rho <\frac{\delta}{\lambda+3}).
\end{align*}
Thus, one obtains (\ref{music}).
Next, we have
\begin{equation} \label{rel2}
\begin{array}{ll}d(y-g(x_n),F(x_n))
&<\rho (\Vert(x_n,y-g(x_n))-(\bar x,\bar y)\Vert +\Vert g(x_n)-g(\bar x)\Vert)\\
&\le \rho (\Vert(x_n,y-g(x_n))-(\bar x,\bar y)\Vert +\lambda \Vert x_n-\bar x\Vert)\\
&\le \rho(\lambda +1)\Vert(x_n,y-g(x_n))-(\bar x,\bar y)\Vert \\&<2\delta\Vert(x_n,y-g(x_n))-(\bar x,\bar y)\Vert  \quad  (\text{since}\; \rho <\frac{\delta}{\lambda+2}).
\end{array}
\end{equation}
So, we receive (\ref{music2}).
By relation (\ref{rel1}), for $n$ sufficiently large, say, $n\ge n_0,$ as $(x_n,y_n)\ne (\bar x,\bar y+g(\bar x))$ we may  find  $t_n>0$ and  \\$(u_n,w_n)\in B( (u,v+Dg(\bar x)(u)),\rho)$ such that
$$ x_n=\bar x+t_nu_n,\quad y=\bar y+g(\bar x)+t_nw_n.$$
Set
$$y-g(x_n)=\bar y+t_nv_n;\;\; v_n=w_n-\frac{g(\bar x+t_nu_n)-g(\bar x)}{t_n}.$$
If $u=0$ then
$$\begin{array}{ll}\Vert (u_n,v_n)-(0,v)\Vert\\& = \Vert(u_n,( v_n-w_n+Dg(\bar x))+(w_n-v-Dg(\bar x)))\Vert\\&\leq \Vert (u_n w_n-v-Dg(\bar x))\Vert +\Vert w_n-v_n-Dg(\bar x)\Vert
\\&\le \rho +\Vert \frac{ g(\bar x +t_n u_n)-g(\bar x)}{t_n} -Dg(\bar x)\Vert
\\&<\rho +\varepsilon<\delta; \end{array}$$
otherwise, $u\not=0,$ since $u_n\in B(u,\rho),$ one has
$$t_n=\frac{\Vert x_n-\bar x\Vert}{\Vert u_n\Vert}\le \frac{\rho}{\Vert u\Vert -\rho}<\eta,$$
and therefore, by (\ref{Diff}),
$$ \begin{array}{ll}\Vert (u_n,v_n)-(u,v)\Vert&\le \Vert (u_n,w_n)-(u,v+Dg(\bar x)(u))\Vert+\Vert w_n-v_n-Dg(\bar x)(u)\Vert\\
&\le \rho+ \left\Vert \frac{g(\bar x+t_nu_n)-g(\bar x)}{t_n}-Dg(\bar x)(u)\right\Vert\le \rho+\varepsilon<\delta .\end{array}$$
Now, pick  $z_n\in F^{-1}(y-g(x_n))$ such that
$$ \Vert x_n-z_n\Vert\le (1+1/n) d(x_n, F^{-1}(y-g(x_n))).$$
Hence, according to (\ref{DMR-F}), (\ref{music}), (\ref{music2}), (\ref{music1}), we obtain,
\begin{equation}\label{TDK}
  \Vert x_n-z_n\Vert\le (1+1/n)\tau d(y-g(x_n),F(x_n)).
\end{equation}
Consequently,
\begin{align*}
\Vert z_n-\bar x\Vert&\leq \Vert x_n-z_n\Vert+\Vert x_n-\bar x\Vert
\\& \leq (1+1/n)\tau\rho\Vert(x_n,y)-(\bar x,\bar y+g(\bar x))\Vert +\Vert x_n-\bar x \Vert
\\ &<(1+1/n)\tau\rho^2+\rho<(1+1/n)\tau\rho+\rho\\&<
\ 2\tau \rho +\rho = \rho(1+2\tau)\\&<\delta.
\end{align*}
Using the local Lipschitz continuity around $\bar x$ of $g$ gives, $\Vert g(z_n)-g(x_n)\Vert\le \lambda \Vert z_n-x_n\Vert.$ As $y-g(x)\notin F(x)$ and $\lim_{n\to\infty}x_n=x,$ then $\liminf_{n\to\infty}\Vert x_n-z_n\Vert>0.$ Hence, by relation (\ref{TDK}), one has
$$ \begin{array}{ll}
|\Gamma\varphi_{F+g}(\cdot,y)|(x)&\ge \limsup_{n\to\infty}\frac{\varphi_{F+g}(x,y)-\varphi_{F+g}(z_n,y)}{\Vert x-z_n\Vert}\\
&\ge \limsup_{n\to\infty}\frac{d(y-g(x_n),F(x_n))-d(y-g(z_n),F(z_n))}{\Vert x_n-z_n\Vert}\\
&\ge  \limsup_{n\to\infty}\frac{d(y-g(x_n),F(x_n))-d(y-g(x_n),F(z_n))-\Vert g(x_n)-g(z_n)\Vert}{\Vert x_n-z_n\Vert}\\
&\ge \limsup_{n\to\infty}\frac{d(y-g(x_n),F(x_n))}{\Vert x_n-z_n\Vert}-\lambda \\
&\ge \limsup_{n\to\infty}\tau^{-1}\frac{n}{n+1}-\lambda= \tau^{-1}-\lambda.
\end{array} $$
Therefore,
$$\liminf_{
\begin{array}{ll}
 &(x,y)\underset {(u,v+Dg(\bar x)(u))} {\longrightarrow}(\bar x,\bar y+g(\bar x))
\\&x\notin (F+g)^{-1}(y), \frac{\varphi_{F+g}(x,y)}{\Vert (x,y)-(\bar x,\bar y)\Vert}\to 0
\end{array}
}
\vert \Gamma \varphi_{F+g}(\cdot,y)\vert (x)\geq \tau^{-1}-\lambda.
$$
Thanks to Theorem \ref{theo1}, the proof is complete.\hfill{$\Box$}\end{proof}
\section{Stability of  Directional H\"older Metric Regularity under Mixed Coderivative-Tangency Conditions  }
For the usual metric regularity, sufficient conditions in terms of   coderivatives have been given by various authors, for instance, in \cite  {Aze06,  JT1,  B.book1,  RefNT3}. In this section, we establish a characterization of  directional H\"{o}lder metric regularity  using the Fr\'{e}chet subdifferential in Asplund spaces,  i.e., in Banach spaces in which every convex continuous function is generically Fr\'{e}chet differentiable. There are  many equivalent descriptions of Asplund spaces, see, e.g., Mordukhovich's book  \cite  {B.book1}
and its bibliography. In particular, any reflexive space is Asplund, as well as each Banach space such that each of its separable subspaces has a separable dual. We shall show that the proposed characterization also ensures the stability of   directional H\"{o}lder metric regularity under suitable differentiable perturbations.
In this section, in order to formulate
 some coderivative characterizations of directional H\"{o}lder metric regularity, some additional
definitions are required.
Let $X$ be  a Banach space. Consider now an extended-real-valued function  $f:X\rightarrow\R\cup\{+\infty\}.$
The {\it Fr\'{e}chet subdifferential} of $f$ at ${\bar x}\in\dom f$
is given as
\begin{eqnarray*}
 \partial f(\bar{x}) = \left\{ x^\ast \in X^\ast: \liminf_{x \rightarrow \bar{x}, \quad x \neq \bar{x}} \frac{f(x) -
f(\bar{x}) - \langle  x^\ast, x - \bar{x} \rangle}{\| x - \bar{x} \|} \ge
0 \right\}.
\end{eqnarray*}
 For the convenience of the reader, we  would like to mention that  the terminology \textit{regular subdifferential}  instead of Fr\'echet subdifferential is also popular   due to its use in Rockafellar and Wets \cite  {Roc-Wet}.
Every element of the Fr\'echet subdifferential  is termed as a Fr\'echet (regular)
subgradient. If $\bar{x} $ is a point where $ f( \bar{x}) = \infty $,
then we set $ \partial f( \bar{x}) = \emptyset $. In fact one can show that an element $  x^\ast $ is a Fr\'echet  subgradient
of $ f $ at $ \bar{x} $ iff
$$
f(x) \ge f(\bar{x}) + \langle  x^\ast, x - \bar{x}\rangle + o (\| x - \bar{x}\|)\quad
\text{where}\quad \displaystyle \lim_{x \rightarrow \bar{x}} \frac{o(\| x - \bar{x}\|)}{\|x - \bar{x}\|}
= 0.$$
It is well-known that the Fr\'{e}chet subdifferential satisfies a fuzzy sum
rule in  Asplund spaces (see \cite  {B.book1}, Theorem 2.33). More precisely, if
$X$ is an Asplund space and $f_{1},f_{2}:X\rightarrow
\mathbb{R\cup\{\infty\}}$ are such that $f_{1}$ is Lipschitz continuous
around $\overline{x}\in\dom f_{1}\cap\dom %
f_{2}$ and $f_{2}$ is lower semicontinuous around $\overline{x},$
then for any $\gamma>0$ one has%
\begin{equation}
\partial(f_{1}+f_{2})(\overline{x})\subset%
{\displaystyle\bigcup}
\{\partial f_{1}(x_{1})+\partial f_{2}%
(x_{2})\mid x_{i}\in\overline{x}+\gamma\overline{B}_{X},\left\vert f
_{i}(x_{i})-f_{i}(\overline{x})\right\vert \leq\gamma,i=1,2\}+\gamma
B_{X^{\ast}}.\label{fuz}%
\end{equation}
For a nonempty closed set $C\subseteq X,$ denote by  $\delta_{C}$ the \textit{indicator function} associated with C
 (i.e. $\delta_{C}(x)=0$,  when $x\in C$ and  $\delta_{C}(x)=\infty$
otherwise). The {\textit{  Fr\'{e}chet normal cone} to $C$  at $\bar x$  is denoted by $N(C,\bar x)$. It  is a closed and convex object in $X^\ast$ which  is defined as
$\partial \delta_C(\bar x).$  Equivalently  a vector $ x^\ast \in X^\ast $ is a Fr\'echet normal to $ C $ at $ \bar{x}$ if
\begin{eqnarray*} \langle x^\ast , x - \bar{x}
\rangle \le o ( \| x - \bar{x} \|), \quad \forall x \in C,
\end{eqnarray*}
where $ \lim_{ x \rightarrow \bar{x}} \displaystyle\frac{o( \| x -
\bar{x} \|)}{\| x - \bar{x}\|} = 0 $.
Let $F:X\rightrightarrows Y$ be a set-valued map and $(x,y)\in\gph F.$ Then the {\it Fr\'{e}chet coderivative} at
$(x,y)$ is the set-valued map $D^{\ast
}F(x,y):Y^{\ast}\rightrightarrows X^{\ast}$ given by
\[
D^{\ast
}F(x,y)(y^{\ast}):=\big\{x^{\ast}\in
X^{\ast}\mid(x^{\ast},-y^{\ast})\in N(\gph%
F,(x,y))\big\}.
\]
\vskip 0.5cm
In the spirit of Gfrerer \cite  {Gfr}, see also Kruger \cite  {Kr15}, Ngai-Tinh \cite  {NT-MOR}  we introduce  the following  {\it limit set for directional H\"{o}lder metric regularity of order $\gamma$} ($\gamma\in (0,1]$).
\begin{definition}
Let  $F: X\rightrightarrows Y$  be   a closed multifunction and $(\bar
x,\bar y)\in\gph F$ be fixed. For  a given $(u,v)\in X\times Y$ and  each $\gamma\in (0,1],$ the {\it critical} limit set  for metric
regularity  of order $\gamma$ of $F$ in the  direction $(u,v)$ at $(\bar x,\bar y)\in \gph F$ is  denoted by
$\mbox{Cr}^\gamma F((\bar x,\bar y),(u,v))$ and  is defined as the set of all
$(w,x^*)\in Y\times X^*$ such that there exist sequences
$\{t_n\}\downarrow 0,$ $\{\varepsilon_n\}\downarrow 0,$ $(u_n,v_n)\in\cone B( (u,v),\varepsilon_n)$ with  $\Vert (u_n,v_n)\Vert=1,$
$w_n\in Y$, $y_n^*\in\mathcal{S}_{Y^*} (\text{the unit sphere of} \;Y^\ast),$\\
$x_n^*\in D^*F(\bar
x+t_nu_n,\bar y+t_nv_n +t_n^{1/\gamma}w_n)(y^*_n)$ with
 $$\frac{\langle
y_n^*,w_n\rangle}{\|w_n\|}\rightarrow 1\quad\mbox{and}\quad \left(w_n, t_n^{(\gamma-1)/\gamma}\Vert w_n\Vert^{\gamma-1}x_n^*\right)\rightarrow (w,x^*).$$
\end{definition}
For a closed multifunction $F:X\rightrightarrows Y,$ we define the {\it coderivative slope} of $F$ at $(x,y)\in\gph F$ as the following quantity:
$$ m_F(x,y)=\inf\left\{\Vert x^*\Vert:\;\; x^*\in D^*F(x,y)(y^*),\; y^*\in\mathcal{S}_Y^*\right\}.$$
\vskip 0.2cm
By applying Theorem \ref{Pertu}, we receive the following result.
\begin{theorem} \label{Coder-Charac} Let  $X$ and $Y$ be Asplund spaces and suppose given a closed multifunction $F:X\rightrightarrows Y$,  as well as  \\ $(\bar
x,\bar y)\in\gph F, (u,v)\in X\times Y$ and $\gamma\in ]0,1[$. If $(0,0)\notin \mbox{Cr}^\gamma F( (\bar x,\bar y),(u,v)),$ then  for any mapping $g:X\rightarrow Y$,  differentiable in a neighborhood of $\bar x,$ and  verifying for some $c\in ]0,1[,$ $\delta>0$
\begin{equation}\label{Compared-slope}
\Vert Dg(x)\Vert\le cm_F(x,y) \quad\forall (x,y)\in\gph F\cap B( (\bar x,\bar y),\delta)\cap \left((\bar x,\bar y)+\cone B( (u,v),\delta)\right),
\end{equation}
the multifunction $F+g$ is directionally H\"{o}lder metrically regular of order $\gamma$ at $(\bar x,\bar y+ g(\bar x))$ in the direction \\ $(u,v+Dg(\bar x)(u)).$ In  particular, $F$ is directionally H\"{o}lder metrically regular of order $\gamma$ at $(\bar x,\bar y)$ in the  direction $(u,v).$
\end{theorem}
\vskip 0.2cm\begin{proof}  gAssume  on the contrary that there is a differentiable function $g:X\rightarrow Y,$ verifying (\ref{Compared-slope}) for some $c\in (0,1)$ and $\delta>0,$ such that  H\"{o}lder metric regularity  of order $\gamma$ at $(\bar x, \bar y+g(\bar x))$ in  the direction $(u,v+Dg(\bar x)(u))$ fails for $F+g$.  According to Theorem \ref{Slope-Chac},   setting $$\varphi(x,y)=\liminf_{u\to x}d(y,F(u)+g(u))=\liminf_{u\to x}d(y-g(x),F(u)),\, (x,y)\in X\times Y,$$
we can find a sequence $\{(x_n,y_n)\}\subseteq X\times Y$ and  a sequence of nonnegative reals $\{\delta_n>0\}$ (we can set $\delta_n= \frac{1}{n}$) such that$$ \varphi(x_n,y_n)>0,\; \delta_n\downarrow 0,\; \Vert x_n- \bar x \Vert<\delta_n,\;  \Vert y_n- \bar y-g(\bar x) \Vert<\delta_n;$$
\begin{equation}\label{HLY}
(x_n,y_n)\in(\bar x.\bar y+g(\bar x))+ \cone B( (u,v+Dg(\bar x)(u)),\delta_n);\end{equation}
$$\lim_{n\to\infty}\frac{\varphi^{\gamma}(x_n,y_n)}{\|(x_n,y_n)-(\bar x,\bar y+g(\bar x))\|}=0,$$
and
$$|\nabla \varphi^\gamma(\cdot,y_n)|(x_n)<\delta_n,\\;\;\forall n\in\N.$$
We can assume that $\delta_n\in (0,\delta/4)$ for all $n\in\N.$ Then,  for each $n\in\N,$ there is
$\eta_n\in (0,\delta_n)$ with
 \begin{equation}\label{para-eta}
\eta_n/\varphi(x_n,y_n)\to 0;\; 2\eta_n<\varphi(x_n,y_n)(1-\delta_n);\;
 \eta_n^2/4+5\eta_n/4<\Vert (x_n,y_n)-(\bar x,\bar y+g(\bar x))\Vert\;\;\forall n\in\N
\end{equation}
 such
that $$d(y_n-g(z),F(z))\geq \varphi(x_n,y_n)(1-\delta_n),\,\forall z\in
B(x_n,4\eta_n),$$  and
$$\delta_n\geq \frac{\varphi^\gamma(x_n,y_n)-\varphi^\gamma(z,y_n)}{\|x_n-z\|}\quad \mbox{for all}\; z\in \overline{ B}(x_n,\eta_n).$$
Equivalently,
$$\varphi^{\gamma}(x_n,y_n)\leq \varphi^\gamma(z,y_n)+\delta_n\|z-x_n\|\quad \mbox{for all}\; z\in \overline{B}(x_n,\eta_n).$$
Since $$\eta_n/\varphi(x_n,y_n)\to 0\quad\mbox{and}\quad \lim_{n\to\infty}\frac{\varphi^{\gamma}(x_n,y_n)}{\|(x_n,y_n)-(\bar x,\bar y+g(\bar x))\|}=0 ,$$ one has
\begin{equation}\label{lim-eta}
\lim_{n\to\infty}\frac{\eta_n^{\gamma}}{\|(x_n,y_n)-(\bar x,\bar y+g(\bar x))\|}=0.
\end{equation}
Take $z_n\in B(x_n,\eta_n^2/4),$ $w_n\in F(z_n)+g(z_n)$ such that
$$\|y_n-w_n\|^\gamma\leq
\varphi(x_n,y_n)^\gamma+\eta_n^2/4.$$
Then,
$$\| y_n-w_n\|^\gamma\leq  \varphi^\gamma(z,y_n)+\delta_n\|z-x_n\|+\eta_n^{2}/4, \;\;\forall z\in \overline{ B}(x_n,\eta_n).$$
By
$$\varphi^\gamma(z,y_n)\le [d(y_n,(F+g)(z))]^\gamma \le \|y_n-w\|^\gamma+\delta_{\mbox{gph}(F+g)}(z,w), \;\;\forall w\in Y.$$
one has
$$\| y_n-w_n\|^\gamma \leq
 \|y_n-w\|^\gamma+\delta_{\mbox{gph}(F+g)}(z,w)+\delta_n\|z-z_n\|+(\delta_n+1)\eta_n^{2}/4$$
for all  $(z,w)\in \overline{B}(x_n,\eta_n)\times Y.$
\vskip 2mm
Applying the
Ekeland variational principle to the function
$$(z,w)\mapsto \| y_n-w\|^\gamma+\delta_{\mbox{gph}(F+g)}(z,w)+\delta_n\|z-z_n\|$$
on $\overline B(x_n,\eta_n)\times Y,$ we can select $(z^1_n,w^1_n)\in
(z_n,w_n)+\frac{\eta_n}{4} B_{X\times Y}$ with $(z^1_n,w^1_n)\in \mbox{gph}(F+g)$
such that \begin{equation}\label{Estim 1}|y_n-w^1_n\|^\gamma+\delta_n\Vert z_n^1-z_n\Vert\leq\|y_n-w_n\|^\gamma (\leq
\varphi^\gamma(x_n,y_n)+\eta_n^{2}/4);\end{equation} and such  that the function
$$(z,w)\mapsto \|y_n-w\|^\gamma+\delta_{\mbox{gph}(F+g)}(z,w)+\delta_n\|z-z_n\|+(\delta_n+1)\eta_n\|(z,w)-(z_n^1,w_n^1)\|$$
attains a minimum on $\overline{B}(x_n,\eta_n)\times Y$ at $(z_n^1,w_n^1).$ Observing that the functions
$$(z,w)\to \|y_n-w\|^\gamma, (z,w)\to\|z-z_n\| \;\text{and}\; (z,w)\to  \|(z,w)-(z_n^1,w_n^1)\|$$ are locally Lipschitz around $(z^1_n,w^1_n),$ thanks to the fuzzy sum rule, we can select
 $$w_n^2\in B_Y(w_n^1,\eta_n);\;(z_n^3,w_n^3)\in
B_{X\times
Y}((z_n^1,w_n^1),\eta_n)\cap\mbox{gph}(F+g);$$
$$w_n^{2*}\in\partial( \|y_n-\cdot\|^\gamma)(w_n^2);\;(z_n^{3*},-w_n^{3*})\in
N(\mbox{gph}(F+g),(z_n^3,w_n^3))$$
satisfying
$$w_n^{3*}=w_n^{2*}+(\delta_n+2)\eta_nw_n^{4*},$$
\begin{equation}\label{Estim 2}
\|w_n^{2*}-w_n^{3*}\|<(\delta_n+2)\eta_n, \end{equation}
 and,
$$ \|z_n^{3*}\|\leq \delta_n +(\delta_n+2)\eta_n.$$
Since  $\|y_n-w_n^2\|\geq
\|y_n-w_n\|-\|w_n^2-w_n\|\geq \varphi(x_n,y_n)(1-\delta_n)-2\eta_n>0$, then
$$w_n^{2*}\in\partial( \| y_n-\cdot\|^\gamma)(w_n^2)=\gamma\|y_n-w_n^2\|^{\gamma-1}\partial(\|y_n-\cdot\|)(w_n^2),$$ and therefore,
 $$w_n^{2*}\\=\gamma\|y_n-w_n^2\|^{\gamma-1}t_n^{2*}$$
with
$\|t_n^{2*}\|=1$ and $\langle t_n^{2*},w_n^2-y_n\rangle=\|y_n-w_n^2\|.$ Thus, from  (\ref{Estim 2}),  it follows that
$$\begin{array}{ll}\|w_n^{3*}\|&\geq \|w_n^{2*}\|-(\delta_n+2)\eta_n
\\&=\gamma\|y_n-w_n^2\|^{\gamma-1}\Vert t_n^{2*}\Vert-(\delta_n+2)\eta_n
\\&=\gamma\|y_n-w_n^2\|^{\gamma-1}-(\delta_n+2)\eta_n>0,\end{array}$$
$$\begin{array}{ll} \|w_n^{3*}\|&\le
 \|w_n^{2*}\|+(\delta_n+2)\eta_n
 \\&=\gamma\|y_n-w_n^2\|^{\gamma-1}\Vert t_n^{2*}\Vert+(\delta_n+2)\eta_n
 \\&=\gamma\|y_n-w_n^2\|^{\gamma-1}+(\delta_n+2)\eta_n.\end{array}$$
Since
$$\begin{array}{ll} t_n&\ge \Vert (x_n,y_n)-(\bar x,\bar y+g(\bar x))\Vert- \Vert z_n^3-x_n\Vert\\&\ge \Vert (x_n,y_n)-(\bar x,\bar y
+g(\bar x))\Vert-\eta_n^2/4-5\eta_n/4>0,\end{array}$$
it makes sense to set
\begin{equation}\label{tron}
t_n=\|(z_n^3,y_n)-(\bar x,\bar y+g(\bar x))\|;\;\; (u_n,v_n)=(z_n^3-\bar x,y_n-\bar y-g(\bar x))/t_n;\;\; \zeta_n=(w_n^3-y_n)/t_n^\frac{1}{\gamma},
\end{equation}
and
$$y_n^*=w_n^{3*}/\|w_n^{3*}\|;\;\;x_n^*=z_n^{3*}/\|w_n^{3*}\|.$$
As $(z_n^3,y_n)\to (\bar x,\bar y+g(\bar x))$ and $(z_n^3,y_n)\not= (\bar x,\bar y+g(\bar x))$ for $n$ sufficiently large, then $(t_n)\downarrow 0$ as $n \to\infty.$
Since
$$\begin{array}{ll}\varphi(x_n,y_n)(1-\delta_n )&\leq  d(y_n,(F+g)(\bar x+t_nu_n))\leq t_n^\frac{1}{\gamma}\|\zeta_n\|
\\&\leq
\|y_n-w_n^1\|+\eta_n\\&\leq \varphi(x_n,y_n)+\eta_n^2/4+\eta_n,
\end{array}$$
one has
$$\|\zeta_n\|\le \frac{\varphi(x_n,y_n)+\eta_n^2/4+\eta_n}{(\|(x_n,y_n)-(\bar x,\bar y+g(\bar x))\|-\eta_n^2/4-5\eta_n/4)^{1/\gamma}}.$$
As $\varphi^{\gamma}(x_n,y_n)/\|(x_n,y_n)-(\bar x,\bar y+g(\bar x))\|\to 0$ as well as $\eta^{\gamma}_n/\Vert (x_n,y_n)-(\bar x,\bar y+g(\bar x))\Vert\to 0,$ one obtains
\begin{equation}\label{v}
\lim_{n\to\infty}\zeta_n=0.
\end{equation}
Next, by using the standard formula for the Fr\'{e}chet coderivative of $F+g$, one has
\begin{equation}
\begin{array}{ll}
x_n^*&\in D^*(F+g)(\bar x+t_nu_n,\bar y+g(\bar x)+t_nv_n+t_n^\frac{1}{\gamma}\zeta_n)(y_n^*)\\
&=D^*F(\bar x+t_nu_n,\bar y+g(\bar x)-g(\bar x+t_nu_n)+t_nv_n+t_n^\frac{1}{\gamma}\zeta_n)(y_n^*)+Dg^*(\bar x+t_nu_n)y_n^*.\notag
\end{array}\end{equation}
 In  other words,
\begin{equation}\label{coder}
z_n^*:=x^*_n-Dg^*(\bar x+t_nu_n)y_n^*\in D^*F(\bar x+t_nu_n,\bar y+g(\bar x)-g(\bar x+t_nu_n)+t_nv_n+t_n^\frac{1}{\gamma}\zeta_n)(y_n^*).
\end{equation}
From (\ref{tron}) and relations
$$(\ref{HLY}),\;\lim_{n\to\infty}\zeta_n=0; g(\bar x+t_nu_n)-g(\bar x)=t_nDg(\bar x)(u_n)+0(t_n), $$

by checking directly, we derive that
\begin{equation}\label{cone}
\left(u_n,\frac{g(\bar x)-g(\bar x+t_nu_n)}{t_n}+v_n \right) \in \cone B( (u,v),\varepsilon_n),
\end{equation}
and
$$ (t_nu_n,g(\bar x)-g(\bar x+t_nu_n)+t_nv_n+t_n^\frac{1}{\gamma}\zeta_n) \in \cone B( (u,v),\varepsilon_n),$$
for some sequence $\{\varepsilon_n\}\downarrow 0.$\\
Therefore, by (\ref{Compared-slope}), one has (for $n$ sufficiently large)
$$ \Vert x_n^*\Vert=\Vert z_n^*+ Dg^*(\bar x+t_nu_n)y_n^*\Vert\ge \Vert z_n^*\Vert-\Vert Dg^*(\bar x+t_nu_n)\Vert\ge (1-c)\Vert z_n^*\Vert.$$
By (\ref{HLY}), one derives that
\begin{equation}\label{Estim3}
\begin{array}{ll}
t_n^{(\gamma-1)/\gamma}\Vert \zeta_n\Vert^{\gamma-1}\|x_n^*\|&=t_n^{(\gamma-1)/\gamma}\Vert \zeta\Vert^{\gamma-1}\|z^{3*}_n\|/\|w^{3*}_n\|\\
&\leq \frac{t_n^{(\gamma-1)/\gamma}\Vert \zeta_n\Vert^{\gamma-1}(\delta_n
+(\delta_n+2)\eta_n)}{\gamma\|y_n-w_n^2\|^{\gamma-1}-(\delta_n+2)\eta_n}\\
&\leq \frac{\|y_n-w_n^3\|^{\gamma-1}(\delta_n
+(\delta_n+2)\eta_n)}{\gamma(\|y_n-w_n^2\|-2\eta_n)^{\gamma-1}-(\delta_n+2)\eta_n}.
\end{array}
\end{equation}
Observing that
$$ \Vert y_n-w_n^2\Vert\ge \Vert y_n-w_n\Vert-\Vert w_n-w_n^2\Vert\ge \varphi(x_n,y_n)-\delta_n-2\eta_n>0,$$
then,
$$ \lim_{n\to\infty}\frac{\eta_n}{\Vert y_n-w_n^2\Vert}\le \lim_{n\to\infty}\frac{\eta_n}{\varphi(x_n,y_n)-\delta_n-2\eta_n}=0. $$

Therefore, (\ref{Estim3}) yields
\begin{equation}\label{limz}
\lim_{n\to\infty}t_n^{(\gamma-1)/\gamma}\Vert \zeta_n\Vert^{\gamma-1}\|z_n^*\|\le\lim_{n\to\infty} (1-c)^{-1} t_n^{(\gamma-1)/\gamma}\Vert \zeta\Vert^{\gamma-1}\|x_n^*\|=0.
\end{equation}
Next, one has
\begin{equation}\notag
\begin{array}{ll}
\displaystyle\frac{\langle y_n^*,\zeta_n\rangle}{\Vert \zeta_n\Vert}&=\displaystyle\frac{\langle y_n^*,w_n^3-y_n\rangle}{\Vert w_n^3-y_n\Vert}\\
&\ge \displaystyle\frac{\langle w_n^{3*},w^3_n-y_n\rangle}{\Vert w_n^{3*}\Vert\Vert w^3_n-y_n\Vert}\\
&\ge \displaystyle\frac{\langle w_n^{2*},w^2_n-y_n\rangle-\Vert w_n^{2*}-w_n^{3*}\Vert\Vert w^2_n-y_n\Vert-\Vert w_n^{2*}\Vert\Vert w^2_n-w^3\Vert}{\Vert w_n^{2*}\Vert\Vert w^2_n-y_n\Vert +\Vert w_n^{2*}\Vert\|w^3_n-w^2_n\|+\|w^3_n-w^2_n\|\Vert w_n^{2*}-w_n^{3*}\Vert+\Vert w_n^{2*}-w_n^{3*}\Vert\Vert w^2_n-y_n\Vert}\\
&\ge\displaystyle \frac{\gamma\Vert w^2_n-y_n\Vert^\gamma -(\delta_n+2)\eta_n(\varphi(x_n,y_n)+\delta_n+2\eta_n)-2\eta_n\gamma\Vert y_n-w_n^2\Vert^{\gamma-1}}{\gamma\Vert w^2_n-y_n\Vert^\gamma+2\eta_n\gamma\Vert y_n-w_n^2\Vert^{\gamma-1}+2\eta^2_n(\delta_n+2)+(\delta_n+2)\eta_n(\varphi(x_n,y_n)+\delta_n+2\eta_n)}.
\end{array}
\end{equation}
Here, we use the following relations
$$\langle w_n^{2*},w^2_n-y_n\rangle= \gamma\Vert w^2_n-y_n\Vert^\gamma;\;\Vert w_n^{2*}\Vert=\gamma\Vert w^2_n-y_n\Vert^{(\gamma-1)}$$
$$ \Vert y_n-w_n^2\Vert\ge \Vert y_n-w_n\Vert-\Vert w_n-w_n^2\Vert\ge \varphi(x_n,y_n)(1-\delta_n)-2\eta_n>0.$$
Hence, since
$$\frac{\eta_n}{\Vert y_n-w_n^2\Vert}\le \frac{\eta_n}{\varphi(x_n,y_n)(1-\delta_n)-2\eta_n}\rightarrow 0,\;\mbox{as}\; n\to\infty,$$
one obtains
\begin{equation}\label{final}
\lim_{n\to\infty}\frac{\langle y_n^*,\zeta_n\rangle}{\Vert \zeta_n\Vert}=1.
\end{equation}
Relations (\ref{v}), (\ref{coder}), (\ref{cone}), (\ref{limz}), (\ref{final}) show that $(0,0)\in \mbox{Cr}^\gamma F( (\bar x,\bar y),(u,v)),$ completing  the proof.\hfill{$\Box$}\end{proof}
\vskip 0.5cm
An open question is whether the inverse implication is true  in general? In  the particular case when  $F$ is a closed multifunction and either $F$ is a closed convex multifunction (i.e., $\gph F$ is closed convex) or $\gamma=1,$ the inverse holds as shown in the following proposition.
\vskip 0.5cm
\begin{proposition} Let $X,Y$ be Banach spaces and let $(u,v)\in X\times Y,$ $\gamma\in (0,1]$ be given.
\begin{itemize}
\item[(i)] If $F:X\rightrightarrows Y$ is a closed convex multifunction, which is directionally H\"{o}lder metrically regular of order $\gamma$ at $(\bar x,\bar y)\in\gph F$ in the direction $(u,v),$ then $(0,0)\notin\mbox{Cr}^\gamma F( (x,y),(u,v));$
\item[(ii)] If a closed multifunction $F:X\rightrightarrows Y$ is directionally metrically regular at $(\bar x,\bar y)\in\gph F$ in  the direction $(u,v),$ then $(0,0)\notin\mbox{Cr}F( (x,y),(u,v)).$

\end{itemize}
\end{proposition}
\vskip 0.2cm
\begin{proof} ({\it i}). Let  us  consider sequences  $\{t_n\}\downarrow 0;$ $\{\varepsilon_n\}\downarrow 0;$ $(u_n,v_n)\in\cone B( (u,v),\varepsilon_n)$ with $\Vert (u_n,v_n)\Vert=1;$ $\zeta_n\in Y;$ $y_n^*\in\mathcal{S}_{Y^*}$ and $x_n^*\in X^*$ such that
$$ \Vert \zeta_n\Vert\to 0;\;\; \frac{\langle y_n^*,\zeta_n\rangle}{\Vert \zeta_n\Vert}\to 1\quad\mbox{as}\; n\to\infty, $$
$$ x_n^*\in D^*F(\bar x+t_nu_n,\bar y+t_nv_n+t_n^{1/\gamma}\zeta_n)(y_n^*).$$
Thanks to the convexity of $F,$
\begin{equation} \label{Convex}
\langle x_n^*,x-\bar x-t_nu_n\rangle- \langle y_n^*,y-\bar y-t_nv_n-t_n^{1/\gamma}\zeta_n\rangle\le 0\;\forall (x,y)\in\gph F.
\end{equation}
Since $F$ is directionally H\"{o}lder metrically regular of order $\gamma$ at $(\bar x,\bar y)\in\gph F$ in  the direction $(u,v),$ with some constant $\tau>0,$ and,
$$ \frac{d(\bar y+t_nv_n, F(\bar x+t_nu_n))}{(t_n\Vert(u_n,v_n)\Vert)^{1/\gamma}}\le \Vert \zeta_n\Vert\rightarrow 0,$$
 for  $n$ sufficiently large, one has
$$ \begin{array}{ll}d(\bar x+t_nu_n, F^{-1}(\bar y+t_nv_n))&\le\tau d(\bar y+t_nv_n, F(\bar x+t_nu_n))^\gamma\\
&\le \tau\Vert \bar y+t_nv_n-\bar y-t_nv_n-t^{1/\gamma}\zeta_n\Vert^\gamma=\tau t_n\Vert \zeta_n\Vert^\gamma.
\end{array} $$
Thus, we can find $x_n\in F^{-1}(\bar y+t_nv_n)$ such that
$$ \Vert x_n-\bar x-t_nu_n\Vert\le \tau(1+t_n) t_n\Vert \zeta_n\Vert^\gamma.$$
Taking into account  (\ref{Convex}), one obtains
$$ \Vert x_n^*\Vert \tau(1+t_n) t_n\Vert \zeta_n\Vert^\gamma\ge \Vert x_n^*\Vert\Vert x_n-\bar x-t_nu_n\Vert\ge\langle x_n^*,\bar x+t_nu_n-x_n\rangle \ge t_n^{1/\gamma}\langle y_n^*,\zeta_n\rangle. $$
Therefore,
$$ \liminf_{n\to\infty} \Vert x_n^*\Vert t_n^{(\gamma-1)/\gamma}\Vert \zeta_n\Vert^{\gamma-1}\ge\lim_{n\to\infty} \tau^{-1}(1+t_n)^{-1}\frac{\langle y_n^*,\zeta_n\rangle}{\Vert \zeta_n\Vert}=\tau^{-1},$$
which shows that $(0,0)\notin\mbox{Cr}^\gamma F( (\bar x,\bar y),(u,v)).$
\vskip 0.2cm
({\it ii}). Suppose that $F$ is  directionally  metrically regular  at $(\bar x,\bar y)\in\gph F$ in  the direction $(u,v).$ Let be given   sequences $\{t_n\}\downarrow 0;$ $\{\varepsilon_n\}\downarrow 0;$ $(u_n,v_n)\in\cone B( (u,v),\varepsilon_n)$ with $\Vert (u_n,v_n)\Vert=1;$ $\zeta_n\in Y;$ $y_n^*\in\mathcal{S}_{Y^*}$ and $x_n^*\in X^*$ such that
$$ \Vert \zeta_n\Vert\to 0;\;\; \frac{\langle y_n^*,\zeta_n\rangle}{\Vert \zeta_n\Vert}\to 1\quad\mbox{as}\; n\to\infty, $$
$$ x_n^*\in D^*F(\bar x+t_nu_n,\bar y+t_nv_n+t_n\zeta_n)(y_n^*).$$ Then, there is a sequence of nonnegative  reals $\{\delta_n\}$ with $\delta_n\in]0,t_n[$ such that
\begin{equation} \label{Subdiff}
\begin{array}{ll}
\langle x_n^*,x-\bar x-t_nu_n\rangle&- \langle y_n^*,y-\bar y-t_nv_n-t_n\zeta_n\rangle\le \varepsilon_n\Vert( (x,y)-(\bar x+t_nu_n,\bar y+t_nv_n+t_n\zeta_n)\Vert\\
&\mbox{for all}\;\; (x,y)\in\gph F\;\;\mbox{with}\; \Vert (x,y)-(\bar x+t_nu_n,\bar y+t_nv_n+t_n\zeta_n)\Vert<\delta_n.
\end{array}
\end{equation}
By setting $y_n=\bar y+t_nv_n+(t_n-\delta_n)\zeta_n,$ then $(\bar x+t_nu_n,y_n)\underset{(u,v)}{\rightarrow}(\bar x,\bar y).$ Hence, by the directional metric regularity of $F$ at $(\bar x,\bar y)$ in  the direction $(u,v)$ with some constant $\tau>0,$ for  $n$  large enough, there is $x_n\in F^{-1}(y_n)$ such that
$$\begin{array}{ll} \Vert x_n-\bar x-t_nu_n\Vert&\le \tau(1+t_n)d(y_n, F(\bar x+t_nu_n))\\
&\le \tau(1+t_n)\Vert y_n- \bar y-t_nv_n-t_n\zeta_n\Vert=\tau(1+t_n)\delta_n\Vert \zeta_n\Vert.
\end{array}$$
Thus, $\Vert (x_n,y_n)-(\bar x+t_nu_n,\bar y+t_nv_n+t_n\zeta_n)\Vert<\delta_n$  for $n$ sufficiently large, and therefore, taking into account of (\ref{Subdiff}), one derives that
$$ \Vert x_n^*\Vert \tau(1+t_n)\delta_n\Vert \zeta_n\Vert\ge\delta_n \langle y_n^*,\zeta_n\rangle,$$
which implies $ \liminf_{n\to\infty} \Vert x_n^*\Vert\ge\tau^{-1},$ and the proof is complete.\hfill{$\Box$}\end{proof}

\section{Applications to the Directional Differentiability of the Optimal Value Function}}
 \vskip 0.5cm
 Let $X,Y$ be  Banach spaces,  $f:X\times Y\to\R$ and $g:X\to Y$ be continuous functions. We suppose that   $K$ is  a closed convex subset of $Y$. In this section, we consider  a parameterized optimization problem of the form
 $$(\mathcal{P}_y)\quad\quad\quad \min_{x\in X}f(x,y)\quad \mbox{s.t.}\quad y\in g(x)-K,$$ depending on
 a  parameter $y\in Y.$ The  \textit{feasible set}  of ($\mathcal{P}_y$) is denoted by
 $$\varPhi(y):=\{x\in X:\quad y\in g(x)-K\}.$$
 For $y=0,$ the corresponding problem ($\mathcal{P}_0$) denoted by ($\mathcal{P}$) is viewed as an unperturbed problem. We set $f(x):=f(x,0)$,  $\varPhi_0:=\varPhi(0)$ and we denote  by $v(y)$ the \textit{optimal value function } of ($\mathcal{P}_y$) and by $\mathcal{S}(y)$ the associated set of optimal solutions:
 $$v(y):=\inf_{x\in\varPhi(y)}f(x,y);$$
 $$\mathcal{S}(y):=\mbox{argmin}_{x\in\varPhi(y)}f(x,y).$$
 Recall that $x_\varepsilon$ is said to be an \textit{$\varepsilon-$optimal solution} of ($\mathcal{P}_y$) if $x_\varepsilon\in\varPhi(y)$ and $f(x_\varepsilon,y)\le v(y)+\varepsilon.$
 In this section, we apply the concept of  directional metric regularity to discuss the directional differentiability of the optimal value function $v(y).$
We now assume that $f(\cdot,\cdot)$ and $g$ are mappings of  class $\mathcal{C}^1.$ Associated to a given direction $d\in Y,$ we consider a path $y(t)$ of the form $y(t)=td+o(t)$, with $t\in\R_+.$
 Let us recall the notion of \textit{  feasible direction}   \cite  {RefBonS}:
 \begin{definition}
 Let $x_0\in \varPhi(0)$ be given. A direction $h\in X$ is said to be a  feasible direction at $x_0$, relative to the direction $d\in Y$, iff for any path $y(t)=td+o(t)$ with $t\ge 0$ in $Y$,  there exists $r(t)=o(t)$ in $X$ such that $x_0+th+r(t)\in\varPhi(y(t))$.
 \end{definition}
 Note  from \cite  {RefBonS} that  if $h$ is a feasible direction relative to $d\in Y$ at $x_0,$ then
 \begin{equation}\label{FD}
 Dg(x_0)h-d\in T_K(g(x_0)),
 \end{equation}
 where, $T_K(g(x_0))=\{d\in Y\,:\, d(g(x_0)  +td, K)=o(t), \quad t\geq 0\}$ stands for the \textit{contingent cone}  to the convex set $K$ at $g(x_0).$
 \vskip 0.2cm
 Conversely, one has the following:
 \begin{lemma}\label{Feasible}
 Let $x_0\in \varPhi(0)$ be given. If relation {\rm(\ref{FD})} holds, and if  in addition, $G(x):=g(x)-K$ is directionally metrically regular in the direction $(h,d)$ at $(x_0,0),$ then $h$ is a feasible direction relative to $d$.
  \end{lemma}
  \vskip 0.2cm
 \begin{proof}  By the assumption, there exist  $\tau,\varepsilon>0$ such that
\begin{equation}\label{aaa}
  \begin{array}{ll}&d(x,\varPhi(y))\le\tau d(y,g(x)-K)\\
  &\forall (x,y)\in B((x_0,0),\varepsilon)\cap[(x_0,0)+\cone B((h,d),\varepsilon)],\;d(y,g(x)-K)\le \varepsilon\Vert (x,y)-(x_0,0)\Vert.
  \end{array}
\end{equation}
 Let $y(t)=td+o(t)$ be given and set $x(t)=x_0+th,$ one has
  $$g(x_0)+tDg(x_0)h-td+o(t)\in K,\;\mbox{as}\;t\downarrow 0.$$
  As $g(x(t))=g(x_0)+tDg(x_0)h+o(t),$ then we have
  $g(x(t))-td+o(t)\in K$ from which we obtain
$$d(y(t),g(x(t))-K)=o(t).$$ Hence, when $t>0$ is sufficiently small,
  $$(x(t),y(t))\in B((x_0,0),\varepsilon)\cap[(x_0,0)+\cone B((h,d),\varepsilon)],\;d(y(t),g(x(t))-K)\le \varepsilon\Vert (x(t),y(t))-(x_0,0)\Vert.$$
  Thus, there is $\bar x(t):=x_0+th+o(t)\in\varPhi(y(t)),$ as $t>0.$\hfill{$\Box$}\end{proof}
  \vskip 0.2cm
  \begin{lemma}\label{Ro-Linearazation} Assume that $Y$ is finite dimensional. Let $x_0\in\varPhi(0)$ and $d\in Y\setminus\{0\},$ $h\in X$ such that (\ref{FD}) holds.
  If $G$ is directionally metrically regular at $(x_0,0)$ in the direction $(h,d),$ then one has
  \begin{equation}\label{RL}
  d\in\inte \{Dg(x_0)X-T_K(g(x_0))\}.
  \end{equation}
  \end{lemma}
  \vskip 0.2cm
 \begin{proof}
  Let $\tau>0,\varepsilon\in(0,1)$ be such that (\ref{aaa}) happens.
Let  $0<\delta<\varepsilon\Vert  d\Vert/2$ and fix  $\tilde{d}\in B(d,\delta)$.
  Since  (\ref{FD}) holds, one has $$g(x_0) + tDg(x_0)h-td+o(t)\in K, \,\text{ as} \,t> 0.$$ Hence,
  $$g(x_0+th)-td+o(t)\in K,\; \,\text{ as}t>0.$$ Moreover, one has $\delta\le \frac{\varepsilon\Vert\tilde{d}\Vert}{2(1-\delta/2)}<\varepsilon \Vert\tilde{d}\Vert $  since $\delta< \varepsilon d\Vert/2<\frac{\varepsilon(\delta+\Vert\tilde{d}\Vert)}{2}.$
  Consequently,  when $t$ is sufficiently small,
 $$d(t\tilde{d},g(x_0+th)-K)\le t\delta+o(t)<\varepsilon t(\Vert h\Vert+\Vert \tilde{d}\Vert).$$
 According to  (\ref{aaa}), we now select  $x(t)\in\varPhi(t\tilde{d})$ such that
 $$\Vert x_0+th-x(t)\Vert\le\tau t\delta+o(t).$$
 Setting $h(t)=\frac{x(t)-x_0}{t},$ one has
 $$\Vert h-h(t)\Vert\le \tau\delta+\frac{o(t)}{t}.$$
 As $x(t)\in\varPhi(t\tilde{d}),$ then
$$ t\tilde{d}\in g(x_0+h(t))-K,$$
and therefore,
$$\tilde{d}\in Dg(x_0)(h(t))+\frac{o(t)}{t}-\frac{K-g(x_0)}{t}.$$
Since $Y$ is finite dimensional, we can take  a sequence $(t_n)_{n\in\N} \downarrow 0$ such that the sequence $(Dg(x_0)h(t_n))_{n\in\N}$ converges to some $w\in Dg(x_0)X.$ Then, thanks to  the preceding relation we obtain
$$\tilde{d} \in Dg(x_0)X-T_K(g(x_0)),$$
which completes  the proof of the lemma.\hfill{$\Box$}\end{proof}
\vskip 0.5cm
Denote by $L(x,\lambda,y)$ and $\Lambda(x_0)$ the \textit{Lagrangian} of ($\mathcal{P}_y$) and  the set of \textit{Lagrange multipliers} of the problem ($\mathcal{P}_0$) for $x_0\in S(0)$, respectively. More precisely,  if $N_K(g(x_0))$ stands for the \textit{normal cone}  to the convex set $K$ at $g(x_0)$, we have:
$$L(x,\lambda,y)=f(x,y)+\langle \lambda,g(x)-y\rangle,\;\; (x,\lambda)\in X\times Y^*;$$
$$\Lambda(x_0)=\{\lambda\in N_K(g(x_0)):\;\; D_xL(x_0,\lambda,0)=0\}.$$
 For a given $d\in Y,$ we now consider the following  \textit{linearization} of ($\mathcal{P}_y$):
$$(\mathcal{P}\mathcal{L}_d)\quad \min_{h\in X}Df(x_0,0)(h,d)\;\;\mbox{s.t.}\;\; Dg(x_0)h-d\in T_K(g(x_0)).$$
From \cite  [p.~ 278]{RefBonS}, we observe that relation (\ref{RL}) is exactly Robinson's constraint qualification for the problem ($\mathcal{P}\mathcal{L}_d$), and the dual of ($\mathcal{P}\mathcal{L}_d$) is
$$(\mathcal{D}\mathcal{L}_d)\quad \max_{\lambda\in\Lambda(x_0)}D_yL(x_0,\lambda,0)d.$$
According to the standard duality  result in extended linear programming (see, e.g., \cite  [Theorem  2.165]{RefBonS}), from Lemma \ref{Ro-Linearazation}, one has the following dual result for the linearization problem ($\mathcal{D}\mathcal{L}_d$):
\begin{lemma}\label{DLi} Let $x_0\in S(0)$ and $d\in Dg(x_0)X-T_K(g(x_0))$  be given. Assume that $Y$ is finite dimensional and  $G=g-K$ is directionally metrically regular at $(x_0,0)$ in  the direction $(h,d)$ for some $h\in X$ with
$Dg(x_0)h-d\in T_K(g(x_0)).$ Then there is no duality gap between problems ($\mathcal{P}\mathcal{L}_d$) and ($\mathcal{D}\mathcal{L}L_d$). Moreover, the common optimal value is finite,  if and only if,  the set $\Lambda(x_0)$ is nonempty; and in this case, the set of optimal solutions of ($\mathcal{D}\mathcal{L}_d$) is a nonempty compact set.
\end{lemma}

\vskip 0.5cm
The following theorem offers  a result related to the Hadamard directional differentiability of the optimal value function $v(y)$.  In order to establish this result, we use the concept of  directional metric regularity  which is weaker than the
 Robinson constraint qualification  in general used in the literature (see, e.g., \cite  [Theorem 4.26]{RefBonS}).
\begin{theorem} Let $Y$ be finite dimensional and let $d\in Y$ with
$$\{d,-d\}\subseteq Dg(x)X-T_K(g(x))\;\;\mbox{for all}\; x\in S(x_0).$$ Assume that
\begin{itemize}
\item[(i)] the multifunction $G=g-K$ is directionally metrically  regular in the  directions $(0,d)$ and $(0,-d)$ at all $(x,0)$ with $x\in S(0);$

\item[(ii)] for any sequence $y_n=t_nd+o(t_n)$ with $t_n\downarrow 0,$ there exists a sequence of $o(t_n)-$optimal solutions $(x_n)$ of ($\mathcal{P}_{y_n}$), converging to some $x_0\in S(0).$
\end{itemize}
Then   denoting  by $v^{'}_{-}(0,d)$ and  $v^{'}_{+}(0,d)$, the lower and upper Hadamard directional derivatives of $v$ at $0$ in the direction $d$,  one has
\begin{equation}\label{HD}
\begin{array}{ll}
&v^{'}_{-}(0,d)\ge\displaystyle\inf_{x\in S(0)\displaystyle}\inf_{\lambda\in\Lambda(x)}D_yL(x,\lambda,0)d;\\
& v_+^{'}(0,d)\le \displaystyle\inf_{x\in S(0)}\displaystyle\sup_{\lambda\in\Lambda(x)}D_yL(x,\lambda,0)d.
\end{array}
\end{equation}
As a result, if $\Lambda(x)$ is  a singleton $\{\lambda(x)\}$ for all $x\in S(0),$ then the Hadamard directional derivative in the  direction $d$ of $v(y)$ at $0$ exists and
$$v^{'}(0,d)=\inf_{x\in S(0)
}D_yL(x,\lambda(x),0)d.$$
\end{theorem}
\vskip 0.2cm
\begin{proof} Let $x\in S(0)$ and let $h\in X$ be  such that $Dg(x)h-d\in T_K(g(x)).$ As $G$ is directionally metrically regular at $(x,0)$ in the direction $(0,d),$ then it  is  also directionally metrically regular at $(x,0)$ in the  direction $(h,d).$ By Lemma \ref{Feasible}, $h$ is a feasible direction relative to $d$, i.e., , $x+th+o(t)\in \varPhi(td)$, $\downarrow  0.$ Therefore,
$$v(td)\le f(x+th+o(t),td)=f(x,0)+tDf(x,0)(h,d)+o(t),$$ and
consequently,
$$\limsup_{t\downarrow  0}\frac{v(td)-v(0)}{t}\le Df(x,0)(h,d).$$
Since $x$ is arbitrary in $S(0)$ and $h$ is an arbitrary feasible point of $(\mathcal{P}\mathcal{L}_d),$ the second inequality in (\ref{HD}) is proved.
\vskip 0.1cm
For the first inequality, let $t_n\downarrow 0;$ $y_n=t_nd+o(t_n)$ and $(x_n)$ be a sequence of $o(t_n)-$solutions of ($\mathcal{P}_{y_n}$) as in (ii), which converges to $x_0\in S(0).$  Pick $h\in X$ such that
$$Dg(x_0)h+d\in T_K(g(x_0));$$
equivalently,
$$g(x_0)+tDg(x_0)h+td+o(t)\in K,\;t\downarrow  0.$$
Since $g(x_n)-y_n\in K,$ and as $K$ is convex, for any $t>0,$ when $n$ is sufficiently large and  such that  $t_n/t<1$, one has
$$\begin{array}{ll}&(1-t_n/t)[g(x_n)-y_n]+t_n/t[g(x_0)+tDg(x_0)h+td+o(t)]\\
&=g(x_n)-y_n+t_n/t[g(x_0)-g(x_n)]+t_nDg(x_0)h +t_nd +t_no(t)/t\in K.
\end{array}$$
Therefore, for $\varepsilon>0,$ when $n$ is sufficiently large, one has $d(g(x_n)+t_nDg(x_0)h,K)\le t_n\varepsilon.
$ Since
$$g(x_n+t_nh)=g(x_n)+t_nDg(x_0)h+o(t_n),$$
one obtains
$$d(g(x_n+t_nh),K)\le t_n\varepsilon+o(t_n).$$
By the directional metric regularity of $G$, we get some  $z_n\in \varPhi_0=G^{-1}(0)$ such that
$x_n+t_nh-z_n=o(t_n\varepsilon).$
Thus,
$$\begin{array}{ll}\displaystyle \frac{v(y_n)-v(0)}{t_n}&\ge\displaystyle  \frac{f(x_n,y_n)-f(z_n,0)+o(t_n)}{t_n}\\
&=\displaystyle\frac{f(x_n,y_n)-f(x_n+t_nh,0)-o(t_n\varepsilon)}{t_n}\\&=-Df(x_0,0)(h,-d)-o(t_n\varepsilon)/t_n.
\end{array}
$$
Finally, by Lemma \ref{DLi}, there is no duality gap between ($\mathcal{P}\mathcal{L}_{-d})$) and ($\mathcal{D}\mathcal{L}_{-d}$); as $\varepsilon>0$ is arbitrary, one derives that
$$\liminf_{n\to\infty}\frac{v(y_n)-v(0)}{t_n}\ge \inf_{\lambda\in\Lambda(x_0)}D_yL(x_0,\lambda,0)d,$$
 from which follows the first inequality of (\ref{HD}).\hfill{$\Box$}\end{proof}

\begin{acknowledgement}
Many thanks to the referees for their helpful comments and suggestions which have led to an improved paper.
\end{acknowledgement}
\section{Conclusions}
In this contribution we have tried to demonstrate how  H\"{o}lder directional metric regularity   of set-valued mappings  is  an  useful concept for studying the stability and the sensitivity analysis of parameterized optimization problems. This has been achieved in the last section,  where  we have investigated  the Hadamard directional differentiability of the optimal value function of a  general parametrized optimization problem.

\section{Conflict of Interest}
The author declares that he has no conflict of interest.

\bibliographystyle{spmpsci_unsrt.bst}
\bibliography{ref_HMR-JOTA}

\begin{thebibliography}{10}
\providecommand{\url}[1]{{#1}}
\providecommand{\urlprefix}{URL }
\expandafter\ifx\csname urlstyle\endcsname\relax
  \providecommand{\doi}[1]{DOI~\discretionary{}{}{}#1}\else
  \providecommand{\doi}{DOI~\discretionary{}{}{}\begingroup
  \urlstyle{rm}\Url}\fi

\bibitem{BorZhuang88}
Borwein, J.M., Zhuang, D.M.: Verifiable necessary and sufficient conditions for
  openness and regularity for set-valued and single-valued maps.
\newblock J. Math. Anal. Appl. \textbf{134}, 441--459 (1988)

\bibitem{Kum-Kl}
Klatte, D., Kummer, B.: Nonsmooth Equations in Optimization, \emph{Nonconvex
  Optimization and its Applications}, vol.~60.
\newblock Kluwer Academic Publishers, Dordrecht (2002).
\newblock Regularity, calculus, methods and applications

\bibitem{DR}
Dontchev, A.L., Rockafellar, R.T.: Implicit Functions and Solution Mappings,
  second edn.
\newblock Springer Series in Operations Research and Financial Engineering.
  Springer, New York (2014)

\bibitem{B.book1}
Mordukhovich, B.S.: Variational Analysis and Generalized Differentiation. { I}
  : { B}asic { T}heory, \emph{Grundlehren der Mathematischen Wissenschaften
  [Fundamental Principles of Mathematical Sciences]}, vol. 330.
\newblock Springer-Verlag, Berlin (2006)

\bibitem{B.book2}
Mordukhovich, B.S.: Variational Analysis and Generalized Differentiation. { II}
  : { { A}pplications}, \emph{Grundlehren der Mathematischen Wissenschaften
  [Fundamental Principles of Mathematical Sciences]}, vol. 331.
\newblock Springer-Verlag, Berlin (2006)

\bibitem{JPP}
Penot, J.P.: Calculus Without Derivatives, \emph{Graduate Texts in
  Mathematics}, vol. 266.
\newblock Springer, New York (2013)

\bibitem{ALEX}
Ioffe, A.: Metric regularity: Theory and applications - a survey.
\newblock In preparation  (2015)

\bibitem{AzeSMAI}
Az\'e, D.: A survey on error bounds for lower semicontinuous functions.
\newblock In: Proceedings of 2003 { MODE} -{ SMAI} { C}onference {J.-P. Penot
  (Ed)}, \emph{ESAIM Proc.}, vol.~13, pp. 1--17. EDP Sci., Les Ulis (2003)

\bibitem{Aze06}
Az\'e, D.: A unified theory for metric regularity of multifunctions.
\newblock J. Convex Anal. \textbf{13}(2), 225--252 (2006)

\bibitem{RefBonS}
Bonnans, J.F., Shapiro, A.: Perturbation Analysis of Optimization Problems.
\newblock Springer Series in Operations Research. Springer-Verlag, New York
  (2000)

\bibitem{BD}
Borwein, J.M., Dontchev, A.L.: On the {B}artle-{G}raves theorem.
\newblock Proc. Amer. Math. Soc. \textbf{131}(8), 2553--2560 (2003)

\bibitem{BorZhu96}
Borwein, J.M., Zhu, Q.J.: Viscosity solutions and viscosity subderivatives in
  smooth { B} anach spaces with applications to metric regularity.
\newblock SIAM J. Contr. Optim. \textbf{34}, 1568--1591 (1996)

\bibitem{RefCom}
Cominetti, R.: Metric regularity, tangent sets, and second-order optimality
  conditions.
\newblock Appl. Math. Optim. \textbf{21}(3), 265--287 (1990)

\bibitem{RefDT1}
Dmitruk, A.V., Kruger, A.Y.: Metric regularity and systems of generalized
  equations.
\newblock J. Math. Anal. Appl. \textbf{342}(2), 864--873 (2008)

\bibitem{DmiKru09.1}
Dmitruk, A.V., Kruger, A.Y.: Extensions of metric regularity.
\newblock Optimization \textbf{58}(5), 561--584 (2009)

\bibitem{F-90}
Frankowska, H.: Some inverse mapping theorems.
\newblock Ann. Inst. H. Poincar\'e Anal. Non Lin\'eaire \textbf{7}(3), 183--234
  (1990)

\bibitem{FQ-12}
Frankowska, H., Quincampoix, M.: H\"older metric regularity of set-valued maps.
\newblock Math. Program. \textbf{132}(1-2, Ser. A), 333--354 (2012)

\bibitem{Io00}
Ioffe, A.D.: Metric regularity and subdifferential calculus.
\newblock Russian Math. Surveys \textbf{55}, 501--558 (2000)

\bibitem{RefIo2}
Ioffe, A.D.: Towards variational analysis in metric spaces: metric regularity
  and fixed points.
\newblock Math. Program., Ser. B \textbf{123}(1), 241--252 (2010)

\bibitem{RefJT}
Jourani, A., Thibault, L.: Metric regularity and subdifferential calculus in {
  B}anach spaces.
\newblock Set-Valued Anal. \textbf{3}(1), 87--100 (1995)

\bibitem{JT1}
Jourani, A., Thibault, L.: Coderivatives of multivalued mappings, locally
  compact cones and metric regularity.
\newblock Nonlinear Anal. \textbf{35}(7), 925--945 (1999)

\bibitem{Lyusternik}
Lyusternik, L.: On conditional extrema of functionals.
\newblock Math. Sbornik \textbf{41}, 390--401 (1934).
\newblock In Russian

\bibitem{RefMorS}
Mordukhovich, B.S., Shao, Y.: Stability of set-valued mappings in infinite
  dimensions: point criteria and applications.
\newblock SIAM J. Control Optim. \textbf{35}(1), 285--314 (1997)

\bibitem{RefNT3}
Huynh, V.N., Th\'era, M.: { Error bounds and implicit multifunction theorem in
  smooth Banach spaces and applications to optimization.}
\newblock { Set-Valued Anal.} \textbf{12}(1-2), 195--223 (2004)

\bibitem{Pen89}
Penot, J.P.: Metric regularity, openness and {L}ipschitz behavior of
  multifunctions.
\newblock Nonlinear Anal. \textbf{13}, 629--643 (1989)

\bibitem{I-Hreg}
Ioffe, A.D.: Nonlinear regularity models.
\newblock Math. Program. \textbf{139}(1-2), 223--242 (2013)

\bibitem{LiMo}
Li, G., Mordukhovich, B.S.: H\"older metric subregularity with applications to
  proximal point method.
\newblock SIAM J. Optim. \textbf{22}(4), 1655--1684 (2012)

\bibitem{OM}
Mordukhovich, B.S., Ouyang, W.: Higher-order metric subregularity and its
  applications.
\newblock J. Global Opt. pp. 1--19 (2015)

\bibitem{AruAvaIzm07}
Arutyunov, A.V., Avakov, E.R., Izmailov, A.F.: Directional regularity and
  metric regularity.
\newblock SIAM J. Optim. \textbf{18}(3), 810--833 (2007)

\bibitem{AI}
Arutyunov, A.V., Izmailov, A.F.: Directional stability theorem and directional
  metric regularity.
\newblock Math. Oper. Res. pp. 526--543 (2006)

\bibitem{I-reg-concept}
Ioffe, A.: On regularity concepts in variational analysis.
\newblock J. Fixed Point Theory Appl. \textbf{8}(2), 339--363 (2010)

\bibitem{Gfrerer-SVA}
Gfrerer, H.: On directional metric subregularity and second-order optimality
  conditions for a class of nonsmooth mathematical programs.
\newblock SIAM J. Optim. \textbf{23}(1), 632--665 (2013)

\bibitem{Gfr2}
Gfrerer, H.: On directional metric regularity, subregularity and optimality
  conditions for nonsmooth mathematical programs.
\newblock Set-Valued Var. Anal. \textbf{21}(2), 151--176 (2013)

\bibitem{Penot-SIOPT}
Penot, J.P.: Second-order conditions for optimization problems with
  constraints.
\newblock SIAM J. Control Optim. \textbf{37}(1), 303--318 (1999)

\bibitem{NTrT}
Ngai, H.V., Nguyen, H.T., Phan, N.T.: Directional {H}\"older metric
  subregularity and application to tangent cones (Preprint)

\bibitem{Roc-Wet}
Rockafellar, R.T.: First and second order epi-differentiability in nonlinear
  programming.
\newblock Trans. Amer. Math. Soc. \textbf{207}, 75--108 (1988)

\bibitem{Kr15.2}
Kruger, A.Y.: Error bounds and {H}\"older metric subregularity.
\newblock arXiv: \textbf{1411.6414} (2015)

\bibitem{Kr15}
Kruger, A.Y.: Error bounds and metric subregularity.
\newblock Optimization \textbf{64}(1), 49--79 (2015)

\bibitem{RefDMT}
Giorgi, E.D., Marino, A., Tosques, M.: Problems of evolution in metric spaces
  and maximal decreasing curve.
\newblock Atti Accad. Naz. Lincei Rend. Cl. Sci. Fis. Mat. Natur. (8)
  \textbf{68}(3), 180--187 (1980)

\bibitem{RefAC2}
Az\'e, D., Corvellec, J.N.: Characterizations of error bounds for lower
  semicontinuous functions on metric spaces.
\newblock ESAIM Control Optim. Calc. Var. \textbf{10}(3), 409--425 (2004)

\bibitem{FabHenKruOut12}
Fabian, M.J., Henrion, R., Kruger, A.Y., Outrata, J.V.: About error bounds in
  metric spaces.
\newblock In: D.~Klatte, H.J. L\"uthi, K.~Schmedders (eds.) Operations Research
  Proceedings 2011. Selected papers of the Int. Conf. Operations Research (OR
  2011), August 30 -- September 2, 2011, Zurich, Switzerland, pp. 33--38.
  Springer-Verlag, Berlin (2012)

\bibitem{RefDoLR}
Dontchev, A.L., Lewis, A.S., Rockafellar, R.T.: The radius of metric
  regularity.
\newblock Trans. Amer. Math. Soc. \textbf{355}(2), 493--517 (2003)

\bibitem{RefDoL}
Dontchev, A., Lewis, A.: Perturbations and metric regularity.
\newblock Set-Valued Anal. \textbf{13}(4), 417--438 (2005)

\bibitem{SiamNT}
Ngai, H.V., Th\'era, M.: Error bounds in metric spaces and application to the
  perturbation stability of metric regularity.
\newblock SIAM J. Optim. \textbf{19}(1), 1--20 (2008)

\bibitem{Gfr}
Gfrerer, H.: First order and second order characterizations of metric
  subregularity and calmness of constraint set mappings.
\newblock SIAM J. Optim. \textbf{21}(4), 1439--1474 (2011)

\bibitem{NT-MOR}
Ngai, H.V., Phan, N.T.: Metric subregularity of multifunctions and
  applications.
\newblock Mathematics of Operations Research To appear in 2015

\end{thebibliography}

\end{document}